\documentclass[11pt]{amsart}

\newtheorem*{srt}{Single Ring Theorem}
\newtheorem*{prop14}{Proposition 14}

\usepackage[top=1.2in, bottom=1.2in, left=1.5in, right=1.5in]{geometry}
\usepackage{amsmath, amsthm, amssymb, enumerate, appendix, color}

\newtheorem{theorem}{Theorem}[section]

\newtheorem{corollary}[theorem]{Corollary}
\newtheorem{lemma}[theorem]{Lemma}

\theoremstyle{remark}
\newtheorem{remark}[theorem]{Remark}

\numberwithin{equation}{section}

\DeclareMathOperator{\E}{\mathbb{E}}

\DeclareMathOperator*{\diag}{diag}

\DeclareMathOperator*{\supp}{supp}

\DeclareMathOperator*{\im}{Im}

\def \N {\mathbb{N}}
\def \P {\mathbb{P}}
\def \R {\mathbb{R}}
\def \C {\mathbb{C}}

\def \T {\mathbb{T}}

\def \one {{\bf 1}}

\def \EE {\mathcal{E}}

\def \e {\varepsilon}

\def \d {\delta}
\def \D {\Delta}

\def \k {\kappa}
\def \l {\lambda}

\def \s {\sigma}
\def \t {\tau}

\def \< {\langle}
\def \> {\rangle}
\def \^ {\widehat}
\def \tran {\mathsf{T}}
\def \HS {\mathrm{HS}}

\renewcommand{\Pr}[2][]{\P_{#1} \left\{ #2 \rule{0mm}{3mm}\right\}}

\newcommand{\til}[1]{\widetilde{#1}}
\newcommand{\ovr}[1]{\check{#1}}

\newcommand{\amp}{\wedge}

\def \smin {s_{\min}}

\title[]{Invertibility of random matrices: unitary and orthogonal perturbations}

\author{Mark Rudelson}
\author{Roman Vershynin}

\address{Department of Mathematics, University of Michigan, 530 Church St., Ann Arbor, MI 48109, U.S.A.}
\email{\{rudelson, romanv\}@umich.edu}
\thanks{M. R. was partially supported by NSF grant DMS 1161372. R. V. was partially supported by NSF grant DMS 1001829.}

\dedicatory{To the memory of Joram Lindenstrauss}

\date{\today}

\subjclass[2000]{60B20}

\begin{document}

\begin{abstract}
  We show that a perturbation of any fixed square matrix $D$ by a random unitary matrix is well invertible with high probability.
  A similar result holds for perturbations by random orthogonal matrices;
  the only notable exception is when $D$ is close to orthogonal.
  As an application, these results completely eliminate a hard-to-check condition from the Single Ring Theorem
  by Guionnet, Krishnapur and Zeitouni.
\end{abstract}

\maketitle

\setcounter{tocdepth}{2}
\tableofcontents

\section{Introduction}

\subsection{The smallest singular values of random matrices}

Singular values capture important metric properties of matrices.
For an $N \times n$ matrix $A$ with real or complex entries, $n \le N$,
the {\em singular values} $s_j(A)$ are the eigenvalues of $|A| = (A^*A)^{1/2}$ arranged in a non-decreasing
order, thus $s_1(A) \ge \ldots s_n(A) \ge 0$.
The smallest and the largest singular values play a special role. $s_1(A)$ is the operator norm of $A$,
while $\smin(A):=s_n(A)$ is the distance in the operator norm
from $A$ to the set of singular matrices (those with rank smaller than $n$).
For square matrices, where $N=n$, the smallest singular value $s_n(A)$ 
provides a {\em quantitative measure of invertibility} of $A$.
It is natural to ask whether typical matrices are well invertible; one often models ``typical'' matrices as random matrices.
This is one of the reasons why the smallest singular values of different classes of random matrices
have been extensively studied (see \cite{RV2} and the references therein).

On a deeper level, questions about the behavior of $\smin(A)$ for random $A$ arise
in several intrinsic problems of random matrix theory.
Quantitative estimates of $\smin(A)$ for square random matrices $A$
with independent entries \cite{R, TV1, RV1, TV2}
were instrumental in proving the Circular Law, which states that the distribution of the eigenvalues of such matrices
converges as $n \to \infty$ to the uniform probability measure on the disc \cite{GT, TV3}.
Quantitative estimates on $\smin(A)$ of random Hermitian matrices $A$ with independent entries above the diagonal
were necessary in the proof of the local semicircle law for the limit spectrum of such matrices \cite{ESY, TV4}.
Stronger bounds for the tail distribution of the smallest singular value of a Hermitian random matrix were established
in \cite{V sym, ESY-rep}, see also \cite{Nguyen}.

\subsection{The main results}

In the present paper we study the smallest singular value for a natural class of random matrices,
namely for random unitary and orthogonal perturbations of a fixed matrix.
Let us consider the complex case first. Let $D$ be any fixed $n \times n$ complex matrix,
and let $U$ be a random matrix uniformly distributed over the unitary group $U(n)$ with respect to the Haar measure.
Then the matrix $D+U$ is non-singular with probability $1$, which can be easily observed considering its determinant.
However, this observation does not give any useful quantitative information on the degree of non-singularity.
A quantitative estimate of the smallest singular value of $D+U$ is one of the two main results of this paper.

\begin{theorem}[Unitary perturbations]					\label{thm: main unit}
  Let $D$ be an arbitrary fixed $n \times n$ matrix, $n \ge 2$.
  Let $U$ be a random matrix uniformly distributed in the unitary group $U(n)$.
  Then
  \[
    \Pr{\smin(D+U) \le t} \le t^c n^C, \quad t > 0.
  \]
\end{theorem}

In the statement above and thereafter $C, c$ denote positive absolute constants.
As a consequence of Theorem~\ref{thm: main unit}, the random matrix $D+U$ is well invertible,
$\|(D+U)^{-1}\| = n^{O(1)}$ with high probability.

\medskip

An important point in Theorem~\ref{thm: main orth} is that the bound is independent of the deterministic matrix $D$.
This feature is essential in the application to the Single Ring Theorem,
which we shall discuss in Section~\ref{s: single ring} below.

\medskip

To see that Theorem~\ref{thm: main orth} is a subtle result, note that in general it {\em fails over the reals}. 
Indeed, suppose $n$ is odd. If $-D, U \in SO(n)$, then $-D^{-1} U \in SO(n)$ has eigenvalue $1$
and as a result $D+U=D(I_n+D^{-1}U)$ is singular.
Therefore, if $D \in O(n)$ is any fixed matrix and $U \in O(n)$ is random uniformly distributed,
$\smin(D+U) = 0$ with probability at least $1/2$.
However, it turns out that this example is essentially the only obstacle for Theorem~\ref{thm: main unit}
in the real case. Indeed, our second main result states that if $D$ is not close to $O(n)$,
then $D+U$ is well invertible with high probability.

\begin{theorem}[Orthogonal perturbations]					\label{thm: main orth}
  Let $D$ be a fixed $n \times n$ real matrix, $n \ge 2$. Assume that
  \begin{equation}				\label{eq: D}
  \|D\| \le K, \quad \inf_{V \in O(n)} \|D-V\| \ge \d
  \end{equation}
  for some $K \ge 1$, $\d \in (0,1)$.
  Let $U$ be a random matrix uniformly distributed in the orthogonal group $O(n)$.
  Then
  $$
  \Pr{\smin(D+U) \le t} \le t^c (Kn/\d)^C, \quad t > 0.
  $$
\end{theorem}
Similarly to the complex case, this bound is uniform over all matrices $D$ satisfying \eqref{eq: D}. This condition is relatively mild: in the case when $K=n^{C_1}$ and $\d=n^{-C_2}$ for some constants $C_1,C_2>0$, we have
  $$
  \Pr{\smin(D+U) \le t} \le t^c n^C, \quad t > 0,
  $$
as in the complex case.
It is possible that the condition $\|D\| \le K$ can be eliminated from the Theorem~\ref{thm: main orth};
we have not tried this in order to keep the argument more readable, and because such condition
already appears in the Single Ring Theorem.

\medskip

Motivated by an application to the Single Ring Theorem,
we shall prove the following more general version of Theorem~\ref{thm: main orth},
which is valid for {\em complex} diagonal matrices $D$.

\begin{theorem}[Orthogonal perturbations, full version]		\label{thm: main orth full}
  Consider a fixed matrix $D = \diag(d_1,\ldots,d_n)$, $n \ge 2$,
  where $d_i \in \C$. Assume that
  \begin{equation}				\label{eq: di}
  \max_{i} |d_i| \le K, \quad
  \max_{i,j} |d_i^2 - d_j^2| \ge\d
  \end{equation}
  for some $K \ge 1$, $\d \in (0,1)$.
  Let $U$ be a random matrix uniformly distributed in the orthogonal group $O(n)$.
  Then
  $$
  \Pr{\smin(D+U) \le t} \le t^c (Kn/\d)^C, \quad t > 0.
  $$
\end{theorem}

Let us show how this result implies Theorem~\ref{thm: main orth}.

\begin{proof}[Proof of Theorem~\ref{thm: main orth} from Theorem~\ref{thm: main orth full}]
Without loss of generality, we can assume that $t \le \d/2$.
Further, using rotation invariance of $U$ we can assume that
$D = \diag(d_1,\ldots,d_n)$ where all $d_i \ge 0$.
The assumptions in \eqref{eq: D} then imply that
\begin{equation}								\label{eq: D diag}
\max_i |d_i| \le K, \quad \max_i |d_i - 1| \ge \d.
\end{equation}

If $\max_{i,j} |d_i^2 - d_j^2| \ge \d^2/4$ then we can finish the proof by
applying Theorem~\ref{thm: main orth full} with $\d^2/4$ instead of $\d$.
In the remaining case we have
\[
 \max_{i,j} |d_i - d_j|^2 \le \max_{i,j} |d_i^2 - d_j^2| < \d^2/4,
\]
which implies
that $\max_{i,j} |d_i - d_j| < \d/2$.
Using \eqref{eq: D diag}, we can choose $i_0$ so that $|d_{i_0} - 1| \ge \d$.
Thus either $d_{i_0} \ge 1 + \d$ or $d_{i_0} \le 1 - \d$ holds.

If $d_{i_0} \ge 1+ \d$ then $d_i > d_{i_0} - \d/2 \ge 1 + \d/2$ for all $i$. In this case
$$
\smin(D+U) \ge \smin(D) - \|U\| > 1 + \d/2 - 1 \ge t,
$$
and the conclusion holds trivially with probability $1$.

If $d_{i_0} \le 1 - \d$ then similarly $d_i < d_{i_0} + \d/2 \le 1 - \d/2$ for all $i$.
In this case
$$
\smin(D+U) \ge \smin(U) - \|D\| > 1 - (1 - \d/2) = \d/2 \ge t,
$$
and the conclusion follows trivially again.
\end{proof}

\subsection{A word about the proofs}

The proofs of Theorems \ref{thm: main unit} and \ref{thm: main orth full} are significantly different
from those of corresponding results for random matrices with i.i.d. entries \cite{R, RV1}
and for symmetric random matrices \cite{V sym}.
The common starting point is the identity $\smin(A)=\min_{x \in S^{n-1}}\|Ax\|_2$.
The critical step of the previous arguments \cite{R, RV1, V sym} was the analysis
of the small ball probability $\Pr{\|Ax\|_2<t}$ for a fixed vector $x \in S^{n-1}$.
The decay of this probability as $t \to 0$ is determined by the {\em arithmetic structure} of the coordinates of the vector $x$.
An elaborate covering argument was used to treat the set of the vectors with a ``bad'' arithmetic structure.
In contrast to this, arithmetic structure plays no role in Theorems \ref{thm: main unit} and \ref{thm: main orth full}.
The difficulty lies elsewhere -- the entries of the matrix $D+U$ are not independent.
This motivates one to seek a way to introduce some independence into the model.
The independent variables have to be chosen in such a way that one can tractably express
the smallest singular value in terms of them. We give an overview of this procedure
in Section~\ref{s: strategy} below.

\medskip

The proof of Theorem~\ref{thm: main orth full} is harder than that of Theorem~\ref{thm: main unit}.
To make the arguments more transparent, the proofs of the two theorems are organized in such a way
that they are essentially self-contained and independent of each other. The reader is encouraged to
start from the proof of Theorem~\ref{thm: main unit}.

\subsection{An application to the Single Ring Theorem}			\label{s: single ring}

The invertibility problem studied in this paper was motivated by a limit law of the random matrix theory,
namely the Single Ring Theorem. This is a result about the eigenvalues of random matrices with prescribed singular values.
The problem was studied by Feinberg and Zee \cite{FZ} on the physical level of rigor,
and mathematically by Guionnet, Krishnapur, and Zeitouni \cite{GKZ}.
Let $D_n=\text{diag}(d_1^{(n)}, \ldots, d_n^{(n)})$ be an $n \times n$ diagonal matrix with non-negative diagonal.
If we choose $U_n$, $V_n$ to be independent, random and uniformly distributed in $U(n)$ or $O(n)$,
then $A_n = U_n D_n V_n$ constitutes the most natural model of a {\em random matrix with prescribed singular values}.
The matrices $D_n$ can be deterministic or random; in the latter case we assume that $U_n$ and $V_n$ are independent of $D_n$.

The Single Ring Theorem \cite{GKZ} describes the typical behavior of the eigenvalues of $A_n$ as $n \to \infty$.
To state this result, we consider the empirical measures of the singular values and the eigenvalues of $A_n$:
\[
  \mu_s^{(n)} := \frac{1}{n} \sum_{j=1}^n \d_{d_j^{(n)}},
  \quad
  \mu_e^{(n)} := \frac{1}{n} \sum_{j=1}^n \d_{\l_j^{(n)}}
\]
where $\d_x$ stands for the $\d$-measure at $x$,
and $\l_1^{(n)}, \ldots, \l_n^{(n)}$ denote the eigenvalues of $A_n$.
Assume that the measures $\mu_s^{(n)}$ converge weakly in probability
to a measure $\mu_s$ compactly supported in $[0,\infty)$.
The Single Ring Theorem \cite{GKZ} states that, under certain conditions,
the empirical measures of the eigenvalues $\mu_e^{(n)}$ converge in probability
to an absolutely continuous rotationally symmetric probability measure $\mu_e$ on $\C$.
Haagerup and Larsen \cite{HL} previously computed the density of $\mu_e$ 
in terms of $\mu_s$ in the context of operator algebras.

In the formulation of this result,
$\s_n(z) := s_n(A_n-zI_n)$ denotes the smallest singular value of the shifted matrix, and
$S_{\mu}$ denotes the Stieltjes transform of a Borel measure $\mu$ on $\R$:
\[
  S_{\mu}(z)= \int_{\R} \frac{d \mu(x)}{z-x}.
\]

\begin{srt} \cite{GKZ}
  Assume that the sequence $\{\mu_s^{(n)} \}_{n=1}^{\infty}$ converges weakly
  to a probability measure $\mu_s$ compactly supported on $\R_+$.
  Assume further:
  \begin{enumerate}[(SR1)]
    \item There exists $M>0$ such that $\Pr{\|D_n\|>M} \to 0$ as $n \to \infty$;
    \item There exist constants $\k, \k_1>0$ such that for any $z \in \C, \ \im(z)> n^{-\kappa}$,
        \[
        \big| \im(S_{\mu_s^{(n)}}(z)) \big| \le \k_1.
        \]
     \item
     There exists a sequence of events $\Omega_n$ with $\P(\Omega_n) \to 1$
     and constants $\d, \d'>0$ such that for Lebesgue almost any $z \in  \C$,
         \[
         \E \big[ \mathbf{1}_{\Omega_n} \mathbf{1}_{\s_n(z)< n^{-\d}} \log^2 \s_n(z) \big] \le \d'.
         \]
  \end{enumerate}
  Then the sequence $\{\mu_e^{(n)} \}_{n=1}^{\infty}$ converges in probability to a probability measure $\mu_e$.
  The measure $\mu_e$ has density, which can be explicitly calculated in terms
  of the measure $\mu_s$\footnote{See \cite[Theorem 4.4]{HL} and \cite[Theorem 1]{GKZ}
  for a precise description of $\mu_e$.}, and whose support coincides with
  a single ring $\{z \in \C: a \le |z| \le b \}$ for some $0 \le a < b < \infty$.
\end{srt}

The explicit formula for the density of the measure $\mu_e$ shows that is strictly positive in the interior of the ring.
This is surprising since the support of the measure $\mu_s$ can have  gaps. Informally, this means that there are no forbidden zones for the eigenvalues, even in the case when there are such zones for singular values.

The inner and outer radii of the ring can be easily calculated \cite{GZ}:
\begin{equation}    \label{eq: radii of the ring}
  a=\left ( \int_0^{\infty} x^{-2} \, d \mu_s(x) \right)^{-1/2}, \quad
  b=\left ( \int_0^{\infty} x^{2} \, d \mu_s(x) \right)^{1/2}.
\end{equation}

\medskip

The first two conditions of the Single Ring Theorem are effectively checkable for a given sequence $d_1^{(n)}, \ldots, d_n^{(n)}$. Indeed, condition (SR1) is readily reformulated in terms of this sequence, since
$\|D_n\|= s_1(D_n)= \max (d_1^{(n)}, \ldots, d_n^{(n)})$.
Condition (SR2) is already formulated in terms of this sequence;
it means that the singular values of the matrices $D_n$ cannot concentrate on short intervals.

Since the relation between the singular values of the original and shifted matrices is not clear,
condition~(SR3) is much harder to check.
It has only been verified in \cite{GKZ} for the original setup of Feinberg and Zee \cite{FZ},
namely for the case when the singular values of $A_n$ are random variables with density
 \[
   f(d_1^{(n)}, \ldots, d_n^{(n)})
   \sim \prod_{j < k} | (d_j^{(n)})^2- (d_k^{(n)})^2|^\beta
    \cdot \exp \Big( -\sum_{j=1}^n P \big(  (d_j^{(n)})^2 \big) \Big)
    \cdot \left( \prod_{j =1}^n d_j^{(n)} \right)^{\beta-1},
 \]
 where $P$ is a polynomial with positive leading coefficient,
 and where $\beta=1$ in the real case and $\beta=2$ in the complex case.
 The proof of condition~(SR3) for this model is based on adding small Gaussian noise
 to the matrix $A_n$ and using coupling to compare
 the eigenvalue distributions of random matrices with and without noise.
 Such approach does not seem to be extendable to more general distributions of singular values.

\medskip

As an application of Theorems \ref{thm: main unit} and \ref{thm: main orth full}, one can show the following:

\begin{corollary}  \label{cor: single ring}
  Condition~(SR3) can always be eliminated from the Single Ring Theorem.
\end{corollary}

The remaining conditions~(SR1) and (SR2)
are formulated in terms of the singular values of the original matrix $D_n$.
This means that the validity of the Single Ring Theorem for a concrete sequence of matrices $D_n$
can now be effectively checked.

\subsection{Organization of the paper}

The rest of this paper is organized as follows.
In Section~\ref{s: strategy} we give an overview and heuristics of the proofs of both main results.
Theorem~\ref{thm: main unit} is proved in Section~\ref{s: proof unit}.
Theorem~\ref{thm: main orth full} is proved in Section~\ref{s: proof orth} with the exception of 
the low dimensions $n=2,3$ that are treated separately in Appendix~\ref{a: main low dim};
some standard tools in the proof of Theorem~\ref{thm: main orth full} 
are isolated in Appendix~\ref{a: tools}.
In Section~\ref{s: single ring theorem} we prove Corollary~\ref{cor: single ring} concerning the Single Ring Theorem.

\subsection{Notation}

We will use the following notation.

Positive absolute constant are denoted $C, C_1, c, c_1, \ldots$; their values may be different in different instances.
The notation $a \lesssim b$ means $a \le C b$ where $C$ is an absolute constant; similarly for $a \gtrsim b$.

The intervals of integers are denoted by
$[n] := \{1, 2, \ldots, n\}$ and $[k:n] := \{k, k+1,\ldots, n\}$ for $k \le n$.

Given a matrix $A$, the minor obtained by removing the first, second and fifth rows and columns of $A$
is denoted $A_{(1,2,5)}$; similarly for other subsets of rows and columns.

The identity matrix on $\R^n$ and $\C^n$ is denoted $I_n$; we often simply write $I$
if the ambient dimension is clear.

Since we will be working with several sources of randomness at the same time,
we denote by $\P_{X,Y} (\EE)$ the conditional probability of the event $\EE$
given all random variables except $X, Y$.

The operator norm of a matrix $A$ is denoted $\|A\|$,
and the Hilbert-Schmidt (Frobenius) norm is denoted $\|A\|_\HS$.

The diagonal matrix with diagonal entries $d_1, \ldots, d_n$ is denoted $\diag(d_1,\ldots,d_n)$.

Finally, without loss of generality we may assume in Theorems~\ref{thm: main unit} and \ref{thm: main orth full}
that $t < c \d$ for an arbitrarily small absolute const $c>0$.

\subsection*{Acknowledgement} The authors are grateful for Ofer~Zeitouni for drawing their attention
to this problem, and for many useful discussions and comments.
The second author learned about the problem at the IMA Workshop oh High Dimensional Phenomena
in September 2011; he is grateful to IMA for the hospitality.
The authors are grateful to Anirban Basak who found an inaccuracy in the earlier version 
of this paper, specifically in the application of Theorem~\ref{thm: main orth} to the Single Ring Theorem 
over reals. Amir Dembo communicated this to the authors, for which they are thankful.

\section{Strategy of the proofs}					\label{s: strategy}

Let us present the heuristics of the proofs of Theorems~\ref{thm: main unit} and \ref{thm: main orth full}.
Both proofs are based on the idea to use local and global structures of the Lie groups $U(n)$ and $O(n)$,
but the argument for $O(n)$ is more difficult.

\subsection{Unitary perturbations}

Our proof of Theorems~\ref{thm: main unit} uses both global and local structures of the Lie group $U(n)$.
The local structure is determined by the infinitesimally small perturbations of the identity in $U(n)$,
which are given by skew-Hermitian matrices.
This allows us to essentially replace $D+U$ (up to $O(\e^2)$ error) by
$$
VD + I  + \e S
$$
where $V$ is random matrix uniformly distributed in $U(n)$,
$S$ is an independent skew-Hermitian matrix, and $\e>0$ is a small number.
The distribution of $S$ can be arbitrary. For example, one 
may choose $S$ to have independent normal above-diagonal entries. (In the actual proof, 
  we populate just one row and column of $S$  by random variables leaving the 
  other entries zero, see \eqref{eq: S}.)
After conditioning on $V$, we are left with a random matrix with a lot of independent entries --
the quality that was missing from the original problem.

However, this local argument is not powerful enough, in particular because real
skew-Hermitian (i.e. skew-symmetric) matrices themselves are singular in odd dimensions $n$.
This forces us to use some global structure of $U(n)$ as well.
A simplest random global rotation is a random complex rotation $R$ in one coordinate in $\C^n$ (given by
multiplication of that coordinate by a random unit complex number).
Thus we can essentially replace $D+U$ by
$$
A = RVD + I + \e S,
$$
and we again condition on $V$.
A combination of the two sources of randomness,
a local perturbation $S$ and a global perturbation $R$,
produces enough power to conclude that $A$ is typically well invertible,
which leads to Theorem~\ref{thm: main unit}.

The formal proof of Theorem~\ref{thm: main unit} is presented in Section~\ref{s: proof unit}.

\subsection{Orthogonal perturbations}

Our proof of Theorem~\ref{thm: main orth full} will also make use of both global and local structures
of the Lie group $O(n)$.
The local structure is determined by the skew-symmetric matrices.
As before, we can use it to
replace $D+U$ by $VD + I  + \e S$ where $V$ is random matrix uniformly distributed in $O(n)$
and $S$ is a random independent Gaussian skew-symmetric
matrix (with i.i.d. $N_\R(0,1)$ above-diagonal entries).

Regarding the global structure, the simplest random global rotation in $O(n)$ is a random rotation $R$
of some {\em two} coordinates in $\R^n$, say the first two.
Still, $R$ alone does not appear to be powerful enough, so we supplement it with
a further {\em random change of basis}. Specifically, we replace $D$
with $\til{D} = QDQ^T$ where $Q$ is a random independent rotation of the first two coordinates.
Overall, we have changed $D+U$ to
$$
\til{A} = R V \til{D} + I + \e S, \quad \text{where} \quad \til{D} = QDQ^T.
$$
Only now do we condition on $V$, and we will work with three sources of randomness --
a local perturbation given by $S$ and two global perturbations given by $R$ and $Q$.

\subsubsection{Decomposition of the problem}
By rotation invariance, we can assume that $D$ is diagonal, thus
$D = \diag(d_1,\ldots,d_n)$.
By assumption, $d_i^2$ and $d_j^2$ are not close to each other for some pair of indices $i,j$;
without loss of generality we can assume that $d_1^2$ and $d_2^2$ are not close to each other.
Recall that our task is to show that
\begin{equation}				\label{eq: inf Ax}
\smin(\til{A}) = \inf_{x \in S^{n-1}} \|\til{A}x\|_2 \gtrsim \e
\end{equation}
with high probability. (In this informal presentation, we suppress the dependence on $n$; it should always be polynomial).
Each $x \in S^{n-1}$ has a coordinate whose magnitude is at least $n^{-1/2}$.
By decomposing the sphere according to which coordinate is large, without loss of generality
we can replace our task \eqref{eq: inf Ax} by showing that
$$
\inf_{x \in S_{1,2}} \|\til{A}x\|_2 \gtrsim \e
$$
where $S_{1,2}$ consists of the vectors $x \in S^{n-1}$ with $|x_1|^2 + |x_2|^2 \ge 1/n$.

In order to use the rotations $R$, $Q$ which act on the first two coordinates, we decompose $\til{A}$ as follows:
\begin{equation}				\label{eq: Atilde decomposed}
\til{A} =
\begin{bmatrix}
A_0 & Y \\
X  & A_{(1,2)}
\end{bmatrix},
\quad \text{where }
A_0 \in \C^{2 \times 2}, \; A_{(1,2)} \in \C^{(n-2) \times (n-2)}.
\end{equation}
We condition on everything except $Q$, $R$ and the first two rows and columns of $S$.
This fixes the minor $A_{(1,2)}$. We will proceed differently depending on whether $A_{(1,2)}$ is well invertible or not.

\subsubsection{When the minor is well invertible}				\label{s: intro well invertible}

Let us assume that
$$
\|M\| \lesssim \frac{1}{\e}, \quad \text{where} \quad M := (A_{(1,2)})^{-1}.
$$
It is not difficult to show (see Lemma~\ref{lem: well invertible minor}) that in this case
$$
\inf_{x \in S_{1,2}} \|\til{A}x\|_2 \gtrsim \e \cdot \smin(A_0 - Y M X).
$$
So our task becomes to prove that
$$
\smin(A_0 - Y M X) \gtrsim 1.
$$
We have reduced our problem to invertibility of a $2 \times 2$ random matrix.

The argument in this case will only rely on the global perturbations $Q$ and $R$
and will not use the local perturbation $S$. So let us assume for simplicity that $S=0$,
although removing $S$ will take some effort in the formal argument.
Expressing the matrix $A_0 - Y M X$ as a function of $R$, we see that
$$
A_0 - Y M X = I + R_0 B
$$
where $B \in \C^{2 \times 2}$ and $R_0 \in O(2)$ is the part of $R$ 
restricted to the first two coordinates (recall that $R$ is identity on the other coordinates).

Note that $I + R_0 B$ has the same distribution as $R_0^{-1} + B$ and $R_0$ is uniformly distributed in $O(2)$.
But invertibility of the latter matrix is the same problem as we are studying in this paper, only in dimension two.
One can prove Theorem~\ref{thm: main orth full} in dimension two (and even for non-diagonal matrices)
by a separate argument based on Remez-type inequalities; see Appendix~\ref{a: main low dim}.
It yields that unless $B$ is approximately {\em complex orthogonal}, i.e. $\|BB^\tran- I\| \ll \|B\|^2$,
the random matrix $I + R_0 B$ is well invertible with high probability in $R_0$,
leading to the desired conclusion. We have thus reduced the problem to showing that $B$
is not approximately complex orthogonal.

To this end we use the remaining source of randomness, the random rotation $Q$. Expressing $b$ as a function of $Q$,
we see that
$$
B = T \til{D}_0
$$
where $T \in \C^{2 \times 2}$ is a fixed matrix, $\til{D}_0 = Q_0 D_0 Q_0^\tran$,
and $Q_0, D_0$ are the $2 \times 2$ minors of $Q$ and $D$ respectively.
Thus $Q_0$ is a random rotation in $SO(2)$ and $D_0 = \diag(d_1, d_2)$.

Now we recall our assumption that $d_1^2$ and $d_2^2$ are not close to each other.
It is fairly easy to show for such $D_0$ that, whatever the matrix $T$ is,
the random matrix $B = T \til{D}_0 = T Q_0 D_0 Q_0^\tran$
is not approximately complex orthogonal with high probability in $Q_0$ (see Lemma~\ref{lem: breaking orthogonality}).
This concludes the argument in this case.
The formal analysis is presented in Section~\ref{s: well invertible}.

\subsubsection{When the minor is poorly invertible}						 \label{s: intro poorly invertible}

The remaining case is when
$$
\|M\| \gg \frac{1}{\e}, \quad \text{where} \quad M := (A_{(1,2)})^{-1}.
$$
We will only use the local perturbation $S$ in this case.

Here we encounter a new problem. Imagine for a moment that the were working with decompositions
into dimensions $1 + (n-1)$ rather than $2 + (n-2)$, thus in \eqref{eq: Atilde decomposed} we had
$A_0 \in \C^{1 \times 1}$,  $A_{(1,2)} \in \C^{(n-1) \times (n-1)}$.
Using the Gaussian random vector $X$,
one could quickly show (see Lemma~\ref{lem: poorly invertible minor}) that in this case
\begin{equation}				\label{eq: intro on S1}
\inf_{x \in S_1} \|\til{A}x\|_2 \gtrsim \e
\end{equation}
with high probability,
where $S_{1}$ consists of the vectors $x \in S^{n-1}$ with $|x_1| \ge n^{-1/2}$.

\medskip

Unfortunately, this kind of argument fails for decompositions into dimensions $2 + (n-2)$
which we are working with.
In other words, we can step {\em one} rather than two dimensions up --
from a poor invertibility of an $(n-1) \times (n-1)$ minor to the good invertibility of the $n \times n$
matrix (on vectors with the large corresponding coordinate). The failure of stepping two dimensions up has a good reason.
Indeed, one can show that Gaussian skew symmetric matrices
are well invertible in even dimensions $n$ and singular in odd dimensions $n$.
Since our argument in the current case only relies on the local perturbation given by a Gaussian skew symmetric matrix, nothing seems to prevent both the $(n-2) \times (n-2)$ minor
and the full $n \times n$ matrix to be poorly invertible if $n$ is odd.

To circumvent this difficulty, we shall redefine the two cases that we have worked with, as follows.

\smallskip

{\bf Case 1:} There exists an $(n-3) \times (n-3)$ minor $A_{(1,2,i)}$ of $A_{(1,2)}$ which is well invertible.
In this case one proceeds by the same argument as in Section~\ref{s: intro well invertible}, but for the decomposition into
dimensions $3 + (n-3)$ rather than $2 + (n-2)$.
The formal argument is presented in Section~\ref{s: well invertible}.

\smallskip

{\bf Case 2:} All $(n-3) \times (n-3)$ minors $A_{(1,2,i)}$ of $A_{(1,2)}$ are poorly invertible.
Let us fix $i$ and apply the reasoning described above, which allows us one to move {\em one} dimension up,
this time from $n-3$ to $n-2$. We conclude that $A_{(1,2)}$ is well invertible on the vectors whose $i$-th coordinate
is large. Doing this for each $i$ and recalling that each vector has at least one large coordinate, we
conclude that $A_{(1,2)}$ is well invertible on all vectors. Now we are in the same situation that we have already analyzed
in Section~\ref{s: intro well invertible}, as the minor $A_{(1,2)}$ is well invertible.
So we proceed by the same argument as there.
The formal analysis of this case is presented in Section~\ref{s: poorly invertible}.

\smallskip

Summarizing, in Case~1 we move three dimensions up, from $n-3$ to $n$, in one step. In Case~2 we make two steps,
first moving one dimension up (from poor invertibility in dimension $n-3$ to good invertibility in dimension $n-2$),
then two dimensions up (from good invertibility in dimension $n-2$ to good invertibility in dimension $n$).

\medskip

This concludes the informal presentation of the proof of Theorem~\ref{thm: main orth full}.

\section{Unitary perturbations: proof of Theorem~\ref{thm: main unit}}    		 \label{s: proof unit}

In this section we give a formal proof of Theorem~\ref{thm: main unit}.

\subsection{Decomposition of the problem; local and global perturbations}	 \label{s: local global}	

\subsubsection{Decomposition of the sphere}

By definition, we have
$$
\smin(D+U) = \inf_{x \in S^{n-1}} \|(D+U)x\|_2.
$$
Since for every $x \in S^{n-1}$ there exists a coordinate $i \in [n]$ such that
$|x_i| \ge 1/\sqrt{n}$, a union bound yields
\begin{equation}				\label{eq: smin via Si}
\Pr{\smin(D+U) \le t}
\le \sum_{i=1}^n \Pr{\inf_{x \in S_i} \|(D+U)x\|_2 \le t}
\end{equation}
where
$$
S_i = \left\{ x \in S^{n-1} :\; |x_i| \ge 1/\sqrt{n} \right\}.
$$
So, without loss of generality, our goal is to bound
\begin{equation}							\label{eq: S1}
\Pr{\inf_{x \in S_1} \|(D+U)x\|_2 \le t}.
\end{equation}

\subsubsection{Introducing local and global perturbations}

We can express $U$ in distribution as
\begin{equation}					\label{U}
U = V^{-1} R^{-1} W
\end{equation}
where $V, R, W \in U(n)$ are random independent matrices, such that
$V$ is uniformly distributed in $U(n)$ while $R$ and $W$ may have arbitrary distributions.
In a moment, we shall choose $R$ as a random diagonal matrix (a ``global perturbation''),
$W$ as a small perturbation of identity with a random skew-Hermitian matrix
(a ``local perturbation''), and we shall then condition on $V$.

So we let $V$ be uniform in $U(n)$ and let
$$
R = \diag(r, 1, \ldots, 1),
$$
where $r$ is a random variable uniformly distributed on the unit torus $\T \subset \C$.
Finally, $W$ will be defined with the help of the following standard lemma.
It expresses quantitatively the local structure of the unitary group $U(n)$,
namely that the tangent space to $U(n)$ at the identity matrix
is given by the skew-Hermitian matrices.

\begin{lemma}[Perturbations of identity in $U(n)$]				\label{lem: identity}
  Let $S$ be an $n \times n$ skew-Hermitian matrix (i.e. $S^* = -S$), let $\e > 0$ and define
  $$
  W_0 = I + \e S.
  $$
  Then there exists $W \in U(n)$ which depends only on $W_0$ and such that
  $$
  \|W - W_0\| \le 2 \e^2 \|S^2\|     \quad \text{whenever} \quad \e^2 \|S^2\| \le 1/4.
  $$
\end{lemma}

\begin{proof}
We write the singular value decomposition $W_0 = U_0 \Sigma V_0$
where $U_0, V_0 \in U(n)$ and $\Sigma$ is diagonal with non-negative entries,
and we define $W := U_0 V_0$.
Since $S$ is skew-Hermitian, we see that
$$
W_0^*W_0 = (I + \e S)^*(I + \e S) = I - \e^2 S^2,
$$
so  $W_0^* W_0 - I = \e^2 S^2$. On the other hand,
the singular value decomposition of $W_0$ yields
$W_0^* W_0 - I = V_0^*(\Sigma^2 - I)V_0$.
Combining these we obtain
$$
\|\Sigma^2 - I\| \le \e^2 \|S^2\|.
$$
Assuming $\e^2 \|S^2\| \le 1/4$ and recalling that $\Sigma$ is a diagonal matrix with non-negative entries
we conclude that
$$
\|\Sigma - I\| \le 2\e^2 \|S^2\|.
$$
It follows that
$$
\|W - W_0\| = \|U_0(I-\Sigma)V_0\| = \|I-\Sigma\| \le 2\e^2 \|S^2\|,
$$
as claimed.
\end{proof}

Now we define the random skew-Hermitian matrix as
\begin{equation}				\label{eq: S}
S =
\begin{bmatrix}
  \sqrt{-1}\, s & -Z^\tran \\
  Z & 0
\end{bmatrix}
\end{equation}
where $s \sim N_\R(0,1)$ and $Z \sim N_\R(0, I_{n-1})$ are independent standard normal random
variable and vector respectively. Clearly, $S$ is skew-Hermitian.

Let $\e \in (0,1)$ be an arbitrary small number. We define $W_0$ and
$W$ as in Lemma~\ref{lem: identity}, and finally we recall that
a random uniform $U$ is represented as in \eqref{U}.

\subsubsection{Replacing $D+U$ by $RVD + I + \e S$}

Let us rewrite the quantity to be estimated \eqref{eq: S1} in terms of the global and local perturbations.
Applying Lemma~\ref{lem: identity} for the random matrix $W_0 = I + \e S$,
we obtain a random matrix $W \in U(n)$, which satisfies the following for every $x \in S_1$:
\begin{align*}
\|(D+U)x\|_2
  &= \|(D + V^{-1} R^{-1} W) x\|_2
  = \|(RVD + W)x\|_2 \\
  &\ge \|(RVD + W_0)x\|_2 - \|W - W_0\| \\
  &\ge \|(RVD + I + \e S)x\|_2 - 2 \e^2 \|S^2\| \quad \text{whenever } \e^2 \|S^2\| \le 1/4.
\end{align*}
Further, $\E \|S\|^2 \le s^2 + 2 \|Z\|_2^2 = 2n-1$, so $\|S\| = O(\sqrt{n})$ with high probability.
More precisely, let $K_0 > 1$ be a parameter to be chosen later, and which satisfies
\begin{equation}							\label{eq: K0 real}
  \e^2 K_0^2 n \le 1/4.
\end{equation}
Consider the event
\begin{equation}							\label{eq: EES}
\EE_S = \big\{ \|S\| \le K_0 \sqrt{n} \big\}; \quad \text{then } \P(\EE_S^c) \le 2 \exp(-c K_0^2 n)
\end{equation}
by a standard large deviation inequality (see e.g. \cite[Corollary 5.17]{V tutorial}).
On $\EE_S$, one has $\e^2 \|S^2\| \le 1/4$ due to \eqref{eq: K0 real}, and thus
$$
\|(D+U)x\|_2 \ge \|(RVD + I + \e S)x\|_2 - 2 \e^2 K_0^2 n.
$$

Denote
$$
A := RVD + I + \e S.
$$
let $\mu \in (0,1)$ be a parameter to be chosen later, and which satisfies
\begin{equation}							\label{eq: mu real}
\mu \ge 2 \e K_0^2 n.
\end{equation}
By the above, our goal is to estimate
\begin{align}				\label{eq: local global}
\Pr{\inf_{x \in S_1} \|(D+U)x\|_2 \le \mu \e}
&\le \Pr{\inf_{x \in S_1} \|Ax\|_2 \le \mu \e + 2 \e K_0^2 n \amp \EE_S} + \P(\EE_S^c) \nonumber\\
&\le \Pr{\inf_{x \in S_1} \|Ax\|_2 \le 2 \mu \e \amp \EE_S} + 2 \exp(-c K_0^2 n).
\end{align}
Summarizing, we now have a control of the first coordinate of $x$, we have introduced the
global perturbation $R$ and the local perturbation $S$, and we replaced
the random matrix $D+U$ by $A = RVD + I + \e S$.

\subsubsection{Decomposition into $1 + (n-1)$ dimensions}

Next, we would like to expose the first row and first column of the matrix $A = RVD + I + \e S$.
We do so first for the matrix
$$
VD =
\begin{bmatrix}
(VD)_{11} & v^\tran \\
u & (VD)_{(1,1)}
\end{bmatrix}
\quad \text{where } u, v \in \C^{n-1}.
$$
Recalling the definition \eqref{eq: S} of $S$, we can express
\begin{equation}				\label{eq: A}
A = RVD + I + \e S =
\begin{bmatrix}
r(VD)_{11} + 1 + \sqrt{-1}\, \e s & (r v - \e Z)^\tran \\
u + \e Z & (I + VD)_{(1,1)}
\end{bmatrix}
=:
\begin{bmatrix}
A_{11} & Y^\tran \\
X & B^\tran
\end{bmatrix}.
\end{equation}
We condition on an arbitrary realization of the random matrix $V$. This fixes
the number $(VD)_{11}$, the vectors $u, v$ and the matrix $B^\tran$ involved in \eqref{eq: A}.
All randomness thus remains in the independent random variables $r$ (which is chosen uniformly in $\T$),
$s \sim N_\R(0,1)$ and the independent random vector $Z \sim N_\R(0, I_{n-1})$.
We regard the random variable $r$ as a global perturbation, and $s, Z$ as local perturbations.

\subsection{Invertibility via quadratic forms}			\label{s:quadratic forms}

Recall from \eqref{eq: local global} that our goal is to bound below the quantity
$$
\inf_{x \in S_1} \|Ax\|_2.
$$
Let $A_1,\ldots,A_n$ denote the columns of $A$. Let $h \in \C^n$ be such that
$$
\|h\|_2 = 1, \quad h^\tran A_i = 0, \quad i=2,\ldots,n.
$$
For every $x \in \C^n$ we have
$$
\|Ax\|_2 = \Big\| \sum_{i=1}^n x_i A_i \Big\|_2
\ge \Big| h^\tran \sum_{i=1}^n x_i A_i \Big|
= |x_1| \cdot |h^\tran A_1|.
$$
Since $|x_1| \ge 1/\sqrt{n}$ for all vectors $x \in S_1$, this yields
\begin{equation}				\label{eq: inf via h}
\inf_{x \in S_1} \|Ax\|_2 \ge \frac{1}{\sqrt{n}}\, |h^\tran A_1|.
\end{equation}

We thus reduced the problem to finding a lower bound on $|h^\tran A_1|$.
Let us express this quantity as a function of $X$ and $Y$ in the decomposition
as in \eqref{eq: A}. The following lemma shows that $|h^\tran A_1|$ is
essentially a quadratic form in $X, Y$, which is ultimately a quadratic form in $Z$.

\begin{lemma}[Quadratic form]					\label{lem: quadratic form}
  Consider an arbitrary square matrix
  $$
  A =
  \begin{bmatrix}
  A_{11} & Y^\tran \\
  X & B^\tran
  \end{bmatrix},
  \quad A_{11} \in \C, \quad X,Y \in \C^{n-1}, \quad B \in \C^{(n-1) \times (n-1)}.
  $$
  Assume that $B$ is invertible. Let $A_1,\ldots,A_n$ denote the columns of $A$. Let $h \in \C^n$ be such that
  $$
  \|h\|_2 = 1, \quad h^\tran A_i = 0, \quad i=2,\ldots,n.
  $$
  Then
  $$
  |h^\tran A_1| = \frac{|A_{11} - X^\tran B^{-1} Y|}{\sqrt{1 + \|B^{-1} Y\|_2^2}}.
  $$
\end{lemma}

\begin{proof}
The argument is from  \cite[Proposition 5.1]{V sym}.
We express $h$ by exposing its first coordinate as
$$
h =
\begin{bmatrix}
h_1 \\ \bar{h}
\end{bmatrix}.
$$
Then
\begin{equation}				\label{eq: h1 hbar}
h^\tran A_1 = (h_1 \; \bar{h}^\tran)
\begin{bmatrix}
A_{11} \\ X
\end{bmatrix}
= A_{11} h_1 + \bar{h}^\tran X
= A_{11} h_1 + X^\tran \bar{h}.
\end{equation}
The assumption that $h^\tran A_i = 0$ for $i \ge 2$ can be stated as
$$
0 = (h_1 \; \bar{h}^\tran)
\begin{bmatrix}
Y^\tran \\ B^\tran
\end{bmatrix}
= h_1 Y^\tran + \bar{h}^\tran B^\tran.
$$
Equivalently,
$h_1 Y + B \bar{h} = 0$.
Hence
\begin{equation}				\label{eq: hbar}
\bar{h} = - h_1 \cdot B^{-1} Y.
\end{equation}
To determine $h_1$, we use the assumption $\|h\|_2 = 1$ which implies
$$
1 = |h_1|^2 + \|\bar{h}\|_2^2 = |h_1|^2 + |h_1|^2 \cdot \|B^{-1} Y\|_2^2.
$$
So
\begin{equation}				\label{eq: h1}
|h_1| = \frac{1}{\sqrt{1 + \|B^{-1} Y\|_2^2}}.
\end{equation}
Combining \eqref{eq: h1 hbar} and \eqref{eq: hbar}, we obtain
$$
|h^\tran A_1| = A_{11} h_1 - h_1 \cdot X^\tran B^{-1} Y.
$$
This and \eqref{eq: h1} complete the proof.
\end{proof}

Let us use this lemma for our random matrix $A$ in \eqref{eq: A}.
One can check that the minor $B$ is invertible almost surely.
To facilitate the notation, denote
$$
M := B^{-1}.
$$
Recall that $M$ is a fixed matrix.
Then
$$
|h^\tran A_1| = \frac{|A_{11} - X^\tran M Y|}{\sqrt{1 + \|M Y\|_2^2}}.
$$

Since as we know from \eqref{eq: A},
$$
A_{11} = r(VD)_{11} + 1 + \sqrt{-1}\, \e s, \quad
X = u + \e Z, \quad
Y = rv - \e Z,
$$
we can expand
\begin{equation}				\label{eq: ratio}
|h^\tran A_1| =
\frac{|r (VD)_{11} + 1 + \sqrt{-1}\, \e s - r u^\tran M v - \e r (Mv)^\tran Z + \e u^\tran M Z + \e^2 Z^\tran M Z|}
  {\sqrt{1 + \|r Mv - \e MZ\|_2^2}}.
\end{equation}
Recall that $(VD)_{11}$, $u$, $v$, $M$ are fixed, while $r$, $s$, $Z$ are random.

Our difficulty in controlling this ratio is that the typical magnitudes of
$\|M\|$ and of $\|Mv\|_2$ are unknown to us. So we shall consider all possible cases depending
on these magnitudes.

\subsection{When the denominator is small}

We start with the case where the denominator in \eqref{eq: ratio} is $O(1)$.
The argument in this case will rely on the local perturbation given by $s$.

Let $K \ge 1$ be a parameter to be chosen later, and let us consider the event
$$
\EE_{\text{denom}} = \left\{ \|r Mv - \e M Z\|_2 \le K \right\}.
$$
This event depends on random variables $r$ and $Z$ and is independent of $s$.
Let us condition on realizations of $r$ and $Z$ which satisfy $\EE_{\text{denom}}$.
We can rewrite \eqref{eq: ratio} as
$$
|h^\tran A_1| \ge
\frac{|r a + \sqrt{-1}\, \e s|}{\sqrt{1 + K^2}}
\ge \frac{|r a + \sqrt{-1}\, \e s|}{2K}
$$
where $a \in \C$ and $r \in \T$ (and of course $K$) are fixed numbers and $s \sim N_\R(0,1)$.
Since the density of $s$ is bounded by $1/\sqrt{2\pi}$, it follows that
$$
\P_s \left\{ |h^\tran A_1| \le \frac{\l \e}{K} \right\} \le C \l, \quad \l \ge 0.
$$
Therefore, a similar bound holds for the unconditional probability:
$$
\Pr{ |h^\tran A_1| \le \frac{\l \e}{K} \text{ and } \EE_{\text{denom}} } \le C \l, \quad \l \ge 0.
$$
Finally, using \eqref{eq: inf via h} this yields
\begin{equation}				\label{eq: denom small}
\Pr{ \inf_{x \in S_1} \|Ax\|_2 \le \frac{\l \e}{K \sqrt{n}} \text{ and } \EE_{\text{denom}} } \le C \l, \quad \l \ge 0.
\end{equation}
This is a desired form of the invertibility estimate, which is useful when the event $\EE_{\text{denom}}$
holds. Next we will analyze the case where it does not.

\subsection{When the denominator is large and $\|M\|$ is small}

If $\EE_{\text{denom}}$ does not occur, either $\|Mv\|_2$ or $\|MZ\|_2$ must be large.
Furthermore, since $M$ is fixed and $Z \sim N_\R(0, I_{n-1})$, we have
$\|MZ\|_2 \sim \|M\|_\HS$ with high probability. We shall consider the cases where
$\|M\|_\HS$ is small and large separately.
In this section we analyze the case where $\|M\|_\HS$ is small.
The argument will rely on a the global perturbation $r$ and the local perturbation $Z$.

To formalize this, assume that $\EE_{\text{denom}}$ does not occur. Then we can estimate
the denominator in \eqref{eq: ratio} as
$$
\sqrt{1 + \|rMv - \e MZ\|_2^2}
\le 2 \|rMv - \e MZ\|_2
\le 2 \|Mv\|_2 + \e \|MZ\|_2.
$$
Note that $\E \|MZ\|_2^2 = \|M\|_\HS^2$. This prompts us to consider the event
$$
\EE_{MZ} := \left\{ \|MZ\|_2 \le K_1 \|M\|_\HS \right\},
$$
where $K_1 \ge 1$ is a parameter to be chosen later. This event is likely. Indeed,
the map $f(Z)=\|MZ\|_2$ defined on $\R^{n-1}$ has Lipschitz norm bounded by $\|M\|$,
so a concentration inequality in the Gauss space (see e.g. \cite[(1.5)]{LT}) implies that
$$
\P(\EE_{MZ}) \ge 1 - \exp \Big( -\frac{c K _1^2 \|M\|_\HS^2}{\|M\|^2} \Big)
\ge 1 - \exp(-c K_1^2).
$$
On the event $\EE_{MZ} \cap \EE_{\text{denom}}^c$ one has
\begin{equation}				\label{eq: K}
\sqrt{1 + \|rMv - \e MZ\|_2^2}  \le 2 \|Mv\|_2 + \e K_1 \|M\|_\HS.
\end{equation}

\bigskip

Now we consider the case where $\|M\|_\HS$ is small. This can be formalized by the event
\begin{equation}							\label{eq: EEM}
\EE_M := \left\{ \|M\|_\HS \le \frac{K}{2 \e K_1} \right\}.
\end{equation}
On the event $\EE_{MZ} \cap \EE_{\text{denom}}^c \cap \EE_M$, the inequality \eqref{eq: K}
yields the following bound on the denominator in \eqref{eq: ratio}:
$$
\sqrt{1 + \|rMv - \e MZ\|_2^2}  \le 2 \|Mv\|_2 + \frac{K}{2}.
$$
On the other hand, the left side of this inequality is at least $K$ by $\EE_{\text{denom}}^c$. Therefore
\begin{equation}				\label{eq: denom dominated}
\sqrt{1 + \|rMv - \e MZ\|_2^2}  \le 4 \|Mv\|_2.
\end{equation}

To estimate the numerator in \eqref{eq: ratio}, let us condition for a moment on all random variables
but $r$. The numerator then takes the form $|ar+b|$ where $a = (VD)_{11} - u^\tran M v - \e (Mv)^\tran Z$
and $b = 1 + \sqrt{-1}\, \e s + \e u^\tran M Z + \e^2 Z^\tran M Z$ are fixed numbers and
$r$ is uniformly distributed in $\T$. A quick calculation yields a general bound on the conditional probability:
$$
\P_r \left\{ |ar+b| \ge \l_1 |a| \right\} \ge 1 - C\l_1, \quad \l_1 \ge 0.
$$
Therefore a similar bound holds unconditionally.
Let $\l_1 \in (0,1)$ be a parameter to be chosen later. We showed that the event
$$
\EE_{\text{num}} :=
  \left\{ \text{numerator in \eqref{eq: ratio}} \ge \l_1 |(VD)_{11} - u^\tran M v - \e (Mv)^\tran Z| \right\}
$$
is likely:
$$
\P(\EE_{\text{num}}) \ge 1 - C\l_1.
$$

Assume that the event $\EE_{\text{num}} \cap \EE_{MZ} \cap \EE_{\text{denom}}^c \cap \EE_M$ occurs.
(Here the first two events are likely, while the other two specify the case being considered
in this section.) We substitute the bounds on the denominator \eqref{eq: denom dominated}
and the numerator (given by the definition of $\EE_{\text{num}}$) into \eqref{eq: ratio} to obtain
$$
|h^\tran A_1|
\ge \frac{\l_1 |(VD)_{11} - u^\tran M v - \e (Mv)^\tran Z|}{4\|Mv\|_2}.
$$
We can rewrite this inequality as
$$
|h^\tran A_1|
\ge \big| d + \frac{\l_1 \e}{4} \cdot w^\tran Z \big|,
\quad \text{where} \quad d = \l_1 \cdot \frac{(VD)_{11} - u^\tran M v}{4\|Mv\|_2},
\quad w :=- \frac{Mv}{\|Mv\|_2}.
$$
Here $d$ is a fixed number and $w$ is a fixed unit vector,
while $Z \sim N_\R(0,I_{n-1})$.
Therefore $w^\tran Z= \theta \gamma$, where $\gamma  \sim N_\R(0, 1)$, and $\theta \in \C, \  |\theta|=1$.
A quick density calculation yields the following bound on the conditional probability
$$
\P_Z \left\{ \big| d + \frac{\l_1 \e}{4} \cdot w^\tran Z \big| \le \l \l_1 \e \right\} \le C\l, \quad \l >0.
$$
Hence a similar bound holds unconditionally:
$$
\Pr{ |h^\tran A_1| \le \l \l_1 \e
  \text{ and } \EE_{\text{num}} \cap \EE_{MZ} \cap \EE_{\text{denom}}^c \cap \EE_M }
\le C\l, \quad \l >0.
$$
Therefore,
\begin{align*}
\Pr{ |h^\tran A_1| \le \l \l_1 \e \text{ and } \EE_{\text{denom}}^c \cap \EE_M }
&\le \P(\EE_{\text{num}}^c) + P(\EE_{MZ}^c) + C\l \\
&\le C\l_1 + \exp(-c K_1^2) + C\l, \quad \l > 0.
\end{align*}
Using \eqref{eq: inf via h}, we conclude that
\begin{align}				\label{eq: denom large M small}
\Pr{ \inf_{x \in S_1} \|Ax\|_2 \le \frac{\l \l_1 \e}{\sqrt{n}} \text{ and } \EE_{\text{denom}}^c \cap \EE_M }
\le C\l_1 + \exp(-c K_1^2) + C\l, \quad \l > 0.
\end{align}

\subsection{When $\|M\|$ is large}					\label{s:M large}

The remaining case to analyze is where $\|M\|$ is large, i.e. where $\EE_M$ does not occur.
Here shall estimate the desired quantity $\inf_{x \in S_1} \|Ax\|_2$ directly, without
using Lemma~\ref{lem: quadratic form}. The local perturbation $Z$ will do the job.

Indeed, on $\EE_M^c$ we have
$$
\|B^{-1}\|
\ge \frac{1}{\sqrt{n}} \|B^{-1}\|_\HS
= \frac{1}{\sqrt{n}} \|M\|_\HS
\ge \frac{K}{2 \e K_1 \sqrt{n}}.
$$
Therefore there exists a vector $\tilde{w} \in \C^{n-1}$ such that
\begin{equation}				\label{eq: wbar}
\|\tilde{w}\|_2 = 1, \quad \|B \tilde{w}\|_2 \le \frac{2 \e K_1 \sqrt{n}}{K}.
\end{equation}
Note that $\tilde{w}$ can be chosen depending only on $B$ and thus is fixed.

Let $x \in S_1$ be arbitrary; we can express it as
\begin{equation}				\label{eq: x1}
x =
\begin{bmatrix}
x_1 \\ \tilde{x}
\end{bmatrix},
\quad \text{where } |x_1| \ge \frac{1}{\sqrt{n}}.
\end{equation}
Set
$$
w =
\begin{bmatrix}
0 \\ \tilde{w}
\end{bmatrix} \in \C^n.
$$
Using the decomposition of $A$ given in \eqref{eq: A}, we obtain
\begin{align*}
\|Ax\|_2
&\ge |w^\tran A x| =
  \left| \begin{bmatrix}
  0 & \tilde{w}^\tran
  \end{bmatrix}
  \begin{bmatrix}
  A_{11} & Y^\tran \\
  X & B^\tran
  \end{bmatrix}
  \begin{bmatrix}
  x_1 \\ \tilde{x}
  \end{bmatrix} \right| \\
&= |x_1 \cdot \tilde{w}^\tran X + \tilde{w}^\tran B^\tran \tilde{x}| \\
&\ge |x_1| \cdot |\tilde{w}^\tran X| - \|B \tilde{w}\|_2 		\quad \text{(by the triangle inequality)} \\
&\ge \frac{1}{\sqrt{n}} \, |\tilde{w}^\tran X| - \frac{2 \e K_1 \sqrt{n}}{K}	
		\quad \text{(using \eqref{eq: x1} and \eqref{eq: wbar}).}
\end{align*}
Recalling from \eqref{eq: A} that $X = u + \e Z$ and taking the infimum over
$x \in S_1$, we obtain
$$
\inf_{x \in S_1} \|Ax\|_2
\ge \frac{1}{\sqrt{n}} \, |\tilde{w}^\tran u + \e \tilde{w}^\tran Z| - \frac{2 \e K_1 \sqrt{n}}{K}.
$$
Recall that $\tilde{w}$, $u$ are fixed vectors, $\|\tilde{w}\|_2 = 1$, and $Z \sim N_\R(0,I_{n-1})$.
Then $ \tilde{w}^\tran Z= \theta \gamma$, where $\gamma \sim N_\R(0,1)$, and $\theta \in \C, \ |\theta|=1$. A quick density calculation yields
the following bound on the conditional probability:
$$
\P_Z \left\{ |\tilde{w}^\tran u + \e \tilde{w}^\tran Z| \le \e \l \right\}
\le C \l, \quad \l > 0.
$$
Therefore, a similar bound holds unconditionally, after intersection with the event $\EE_M^c$.
So we conclude that
\begin{equation}				\label{eq: M large}
\Pr{ \inf_{x \in S_1} \|Ax\|_2 \le \frac{\e \l}{\sqrt{n}} - \frac{2 \e K_1 \sqrt{n}}{K}
  \text{ and } \EE_M^c }
\le C \l, \quad \l > 0.
\end{equation}

\subsection{Combining the three cases}		\label{s:combining}

We have obtained lower bounds on $\inf_{x \in S_1} \|Ax\|_2$ separately in
each possible case:
\begin{itemize}
  \item inequality \eqref{eq: denom small} in the case of small denominator (event $\EE_{\text{denom}}$);
  \item inequality \eqref{eq: denom large M small} in the case of large denominator, small $\|M\|$
  (event $\EE_{\text{denom}}^c \cap \EE_M$);
  \item inequality \eqref{eq: M large} in the case of large $\|M\|$ (event $\EE_M^c$).
\end{itemize}

To combine these three inequalities, we set
$$
\mu := \frac{1}{2}
  \min \left( \frac{\l}{K \sqrt{n}}, \; \frac{\l \l_1}{\sqrt{n}}, \; \frac{\l}{\sqrt{n}} - \frac{2 K_1 \sqrt{n}}{K} \right).
$$
We conclude that if the condition \eqref{eq: K0 real} on $K_0$
and the condition \eqref{eq: mu real} on $\mu$ are satisfied, then
\begin{align*}
\Pr{ \inf_{x \in S_1} \|Ax\|_2 \le 2 \mu \e }
  &\le C\l + (C \l_1 + \exp(-cK_1^2) + C\l) + C\l \\
  &= 3 C \l + C \l_1 + \exp(-cK_1^2).
\end{align*}
Substituting into \eqref{eq: local global}, we obtain
$$
\Pr{\inf_{x \in S_1} \|(D+U)x\|_2 \le \mu \e}
\le 3 C \l + C \l_1 + \exp(-cK_1^2) + 2 \exp(-c K_0^2 n).
$$
The same holds for each $S_i$, $i \in [n]$.
Substituting into \eqref{eq: smin via Si}, we get
$$
\Pr{\smin(D+U) \le \mu \e}
\le 3 C \l n + C \l_1 n + \exp(-cK_1^2) n + 2 \exp(-c K_0^2 n) n.
$$
This estimate holds for all $\l, \l_1, \e \in (0,1)$ and all $K, K_0, K_1 \ge 1$
provided that the conditions \eqref{eq: K0 real} on $K_0$
and \eqref{eq: mu real} on $\mu$ are satisfied.
So for a given $\e \in (0,1)$, let us choose
$$
\l = \l_1 = \e^{0.1}, \quad
K_0 = K_1 = \log(1/\e), \quad
K = \frac{4 K_1 n}{\l} = 4 \log(1/\e) n \e^{-0.1}.
$$
Then
$$
\mu \gtrsim \frac{\l \l_1}{K \sqrt{n}} = \frac{\e^{0.3}}{4 \log(1/\e) n^{3/2}}.
$$
Assume that $\e \le c' n^{-4}$ for a sufficiently small absolute constant $c'>0$; then
one quickly checks that the conditions \eqref{eq: K0 real} and \eqref{eq: mu real} are satisfied.
For any such $\e$ and for the choice of parameters made above, our conclusion becomes
$$
\Pr{ \smin(D+U) \le \frac{\e^{0.3}}{4 \log(1/\e) n^{3/2}} }
\lesssim \e^{0.1} n + \exp(-c \log^2(1/\e)) + \exp(-c \log^2(1/\e) n) n.
$$
Since this estimate is valid for all $\e \le c' n^{-4}$,
this quickly leads to the conclusion of Theorem~\ref{thm: main unit}. \qed

\section{Orthogonal perturbations: proof of Theorem~\ref{thm: main orth full}}     		 \label{s: proof orth}

In this section we give a formal proof of Theorem~\ref{thm: main orth full}.

\subsection{Initial reductions of the problem}

\subsubsection{Eliminating dimensions $n=2,3$}

Since our argument will make use of $(n-3) \times (n-3)$ minors, we would like
to assume that $n > 3$ from now on. This calls for a separate argument in dimensions $n=2,3$.
The following is a somewhat stronger version Theorem~\ref{thm: main orth full} in these dimensions.

\begin{theorem}[Orthogonal perturbations in low dimensions]		\label{thm: main low dim}	
  Let $B$ be a fixed $n \times n$ complex matrix, where $n \in \{2,3\}$.
  Assume that
  \begin{equation}				\label{eq: no complex orthogonality}
  \|B B^\tran - I\| \ge \d \|B\|^2
  \end{equation}
  for some $\d \in (0,1)$.
  Let $U$ be a random matrix uniformly distributed in $O(n)$.
  Then
  $$
  \Pr{\smin(B+U) \le t} \le C (t/\d)^c, \quad t>0.
  $$
\end{theorem}

We defer the proof of Theorem~\ref{thm: main low dim} to Appendix~\ref{a: main low dim}.

\medskip

Theorem~\ref{thm: main low dim} readily implies Theorem~\ref{thm: main orth full} in dimensions $n=2, 3$.
Indeed, the assumptions in \eqref{eq: di} yield
$$
\|D D^\tran - I\|
= \max_i |d_i^2 - 1|
\ge \frac{1}{2} \max_{i, j} |d_i^2 - d_j^2|
\ge \frac{\d}{2}
\ge \frac{\d}{2K^2} \, \|D\|^2.
$$
Therefore we can apply Theorem~\ref{thm: main low dim}
with $\d/2K^2$ instead of $\d$, and obtain
$$
\Pr{\smin(D+U) \le t} \le C (2K^2 t / \d)^c, \quad t>0.
$$
Thus we conclude a slightly stronger version of Theorem~\ref{thm: main orth full}
in dimensions $n=2, 3$.

\begin{remark}[Complex orthogonality]
  The factor $\|B\|^2$ can not be removed from the assumption \eqref{eq: no complex orthogonality}.
  Indeed, it can happen that $\|B B^\tran - I\| = 1$ while $B+U$ is arbitrarily poorly invertible.
  Such an example is given by the matrix
  $B = M \cdot
    \left[\begin{smallmatrix}
      1 & i \\
      i & -1
     \end{smallmatrix}\right]$
  where $M \to \infty$.
  Then $\det(B+U) = 1$ for all $U \in O(2)$; $\|B+U\| \sim M$, thus
  $\smin(B+U) \lesssim 1/M \to 0$.
  On the other hand, $B B^\tran = 0$.

  This example shows that, surprisingly, staying away from the set of complex orthogonal matrices
  at (any) {\em constant} distance may not guarantee good invertibility of $B+U$.
  It is worthwhile to note that this difficulty does not arise for real matrices $B$.
  For such matrices one can show that factor $\|B\|^2$ can be removed from \eqref{eq: no complex orthogonality}.
\end{remark}

\begin{remark}
Theorem~\ref{thm: main low dim} will be used not only to eliminate the low dimensions $n=2$ and $n=3$
in the beginning of the argument. We will use it one more time in the heart of the proof,
in Subsection~\ref{sub: in dimension 3}
where the problem in higher dimensions $n$ will get reduced to invertibility of certain matrices in 
dimensions $n=2,3$.
\end{remark}

\subsection{Local perturbations and decomposition of the problem}

We can represent $U$ in Theorem~\ref{thm: main orth full} as
$U = V^{-1} W$ where $V,W \in O(n)$ are random independent matrices,
$V$ is uniformly distributed in $O(n)$ while $W$ may have arbitrary distribution.
We are going to define $W$ as a small random perturbation of identity.

\subsubsection{The local perturbation $S$}				\label{s: local perturbation}

Let $S$ be an independent random Gaussian skew-symmetric matrix;
thus the above-diagonal entries of $S$ are i.i.d. $N_\R(0,1)$ random variables and $S^\tran = -S$.
By Lemma~\ref{lem: identity}, $W_0 = I + \e S$ is approximately orthogonal up to error $O(\e^2)$.
Although this lemma was stated over the complex numbers
it is evident from the proof that the same result holds over the reals as well
(skew-Hermitian is replaced by skew-symmetric, and $U(n)$ by $O(n)$).

More formally, fix an arbitrary number $\e \in (0,1)$.
Applying the real analog of Lemma~\ref{lem: identity}
for the random matrix $W_0 = I + \e S$, we obtain a random matrix $W \in O(n)$
that satisfies
\begin{align*}
\smin(D+U)
  &= \smin(D + V^{-1} W) = \smin(V D + W) \\
  &\ge \smin(V D + W_0) - \|W - W_0\| \\
  &\ge \smin(V D + W_0) - 2 \e^2 \|S^2\|
    \quad \text{whenever} \quad \e^2 \|S^2\| \le 1/4.
\end{align*}

Further, $\|S\| = O(\sqrt{n})$ with high probability.
Indeed, let $K_0 > 1$ be a parameter to be chosen later, which satisfies
\begin{equation}				\label{eq: K0}
\e^2 K_0^2 n \le 1/4.
\end{equation}
Consider the event
\begin{equation}				\label{eq: S norm}
\EE_S := \bigl\{ \|S\| \le K_0 \sqrt{n} \bigr\};
\quad \text{then }
\P(\EE_S^c) \le 2 \exp(-c K_0^2 n)
\end{equation}
 provided that
\begin{equation}				\label{eq: C0}
K_0 >C_0
\end{equation}
 for an appropriately large constant $C_0$. 
Indeed, by rotation invariance $S$ has the same distribution as $(\hat{S} - \hat{S}^\tran)/\sqrt{2}$
where $\hat{S}$ is the matrix with all independent $N_\R(0,1)$ entries.
But for the matrix $\hat{S}$, a version of \eqref{eq: S norm} is 
a standard result on random matrices with iid entries, see \cite[Theorem~5.39]{V tutorial}.
Thus by triangle inequality, \eqref{eq: S norm} holds also for $S$.

On $\EE_S$, one has $\e^2 \|S^2\| \le 1/4$ due to \eqref{eq: K0}, and thus
$$
\smin(D+U) \ge \smin(V D + W_0) - 2 \e^2 K_0^2 n.
$$

Next, let $\mu \in (0,1)$ be a parameter to be chosen later, and which satisfies
\begin{equation}   \label{eq: mu is big}
\mu \ge 2 \e K_0^2 n.
\end{equation}
Our ultimate goal will be to estimate
$$
p := \Pr{\smin(D+U) \le \mu \e}.
$$
By the above, we have
\begin{align}					\label{eq: p}
p
  &\le \Pr{\smin(V D + W_0) \le \mu \e + 2 \e^2 K_0^2 n \amp \EE_S} + \P(\EE_S^c) \nonumber \\
  &\le \Pr{\smin(V D + W_0) \le 2 \mu \e \amp \EE_S} + \P(\EE_S^c) \nonumber\\
  &\le \Pr{\smin(A) \le 2 \mu \e \amp \EE_S} + 2 \exp(-c K_0^2 n),
\end{align}
where
$$
A := V D + I + \e S.
$$
Summarizing, we have introduced a local perturbation $S$,
which we can assume to be well bounded due to $\EE_S$.
Moreover, $S$ is independent from $V$, which is uniformly distributed in $O(n)$.

\subsubsection{Decomposition of the problem}

We are trying to bound below
$$
\smin(A) = \inf_{x \in S^{n-1}} \|Ax\|_2.
$$
Our immediate task is to reduce the set of vectors $x$ in the infimum to those
with considerable energy in the first two coordinates, $|x_1|^2 + |x_2|^2 \ge 1/n$.
This this allow us to introduce global perturbations $R$ and $Q$, which
will be rotations of the first few (two or three) coordinates.

To this end, note that for every $x \in S^{n-1}$ there exists a coordinate $i \in [n]$
such that $|x_i| \ge n^{-1/2}$. Therefore
\begin{equation}				\label{eq: sphere decomposition}
S^{n-1} = \bigcup_{i \in [n]} S_i
\quad \text{where } S_i := \bigl\{ x \in S^{n-1} :\; |x_i| \ge n^{-1/2} \bigr\}.
\end{equation}
More generally, given a subset of indices $J \in [n]$, we shall work with the set of vectors
with considerable energy on $J$:
$$
S_{J} := \biggl\{ x \in S^{n-1} :\; \sum_{j \in J} x_j^2 \ge 1/n \biggr\}.
$$
Note that the sets $S_J$ increase by inclusion:
$$
J_1 \subseteq J_2 \quad \text{implies} \quad S_{J_1} \subseteq S_{J_2}.
$$
To simplify the notation, we shall write $S_{1,2,3}$ instead of $S_{\{1,2,3\}}$, etc.

Using \eqref{eq: sphere decomposition}, we decompose the event we are trying to estimate as follows:
$$
\bigl\{ \smin(A) \le 2 \mu \e \bigr\}
= \bigcup_{i \in [n]} \biggl\{ \inf_{x \in S_i} \|Ax\|_2 \le 2 \mu \e \biggr\}.
$$

Next, for every $i \in [n]$,
\[
 \max_{j \in [n]} |d_i^2-d_j|^2 \ge \frac{1}{2} \, \max_{i, j \in [n]} |d_i^2-d_j|^2,
\]
so the second assumption in \eqref{eq: D diag} implies
that  there exists $j=j(i) \in [n]$, $j \ne i$, such that
$|d_i^2 - d_{j(i)}^2| \ge \d$.

Since $S_i \subseteq S_{i,j(i)}$, we obtain from the above and \eqref{eq: p} that
\begin{align}					\label{eq: sum pi}
p
  &\le \sum_{i=1}^n \Pr{ \inf_{x \in S_{i,j(i)}} \|Ax\|_2 \le 2 \mu \e \amp \EE_S } + 2 \exp(-c K_0^2 n) \nonumber\\
  &=: \sum_{i=1}^n p_i + 2 \exp(-c K_0^2 n).
\end{align}
We reduced the problem to estimating each term $p_i$.
This task is similar for each $i$, so without loss of generality we can focus on $i=1$.
Furthermore, without loss of generality we can assume that $j(1) = 2$. Thus our goal is to estimate
$$
p_1 = \Pr{ \inf_{x \in S_{1,2}} \|Ax\|_2 \le 2 \mu \e \amp \EE_S }
$$
under the assumption that

\begin{equation}				\label{eq: di assumption}
|d_1^2 - d_2^2| \ge \d.
\end{equation}

\medskip

Finally, we further decompose the problem according to whether there exists a well invertible $(n-3) \times (n-3)$ minor
of $A_{(1,2)}$ or not. Why we need to consider these cases was explained informally in Section~\ref{s: intro poorly invertible}.

Let $K_1 \ge 1$ be a parameter to be chosen later. By a union bound, we have
\begin{align}				\label{eq: sum p1i}
p_1
  &\le \sum_{i=3}^n \Pr{ \inf_{x \in S_{1,2}} \|Ax\|_2 \le 2 \mu \e \amp \|(A_{(1,2,i)})^{-1}\| \le \frac{K_1}{\e} \amp \EE_S }
   \nonumber\\
  &\phantom{xxxx} + \Pr{ \inf_{x \in S_{1,2}} \|Ax\|_2 \le 2 \mu \e \amp \|(A_{(1,2,i)})^{-1}\| > \frac{K_1}{\e}
    \ \forall i \in [3:n] \amp \EE_S } \nonumber\\
  &=: \sum_{i=3}^n p_{1,i} + p_{1,0}.
\end{align}

\subsection{When a minor is well invertible: going $3$ dimensions up}		 \label{s: well invertible}

In this section we shall estimate the probabilities $p_{1,i}$, $i=3,\ldots,n$, in the decomposition \eqref{eq: sum p1i}.
All of them are similar, so without loss of generality we can focus on estimating $p_{1,3}$.
Since $S_{1,2} \subseteq S_{1,2,3}$, we have
\begin{equation}				\label{eq: p13}
p_{1,3} \le \Pr{ \inf_{x \in S_{1,2,3}} \|Ax\|_2 \le 2 \mu \e \amp \|(A_{(1,2,3)})^{-1}\| \le \frac{K_1}{\e} \amp \EE_S }.
\end{equation}
This is the same as the original invertibility problem, except now we have three extra pieces of information:
(a) the minor $A_{(1,2,3)}$ is well invertible;
(b) the vectors in $S_{1,2,3}$ over which we are proving invertibility have large energy in the first three coordinates;
(c) the local perturbation $S$ is well bounded.

\subsubsection{The global perturbations $Q$, $R$}

The core argument in this case will rely on global perturbations (rather than the local perturbation $S$),
which we shall now introduce into the matrix $A = VD + I + \e S$.
Define
$$
Q :=
\begin{bmatrix}
Q_0 & 0 \\
0 & \ovr{I}
\end{bmatrix},
\quad
R :=
\begin{bmatrix}
R_0 & 0 \\
0 & \ovr{I}
\end{bmatrix}
$$
where $Q_0 \in SO(3)$ and $R_0 \in O(3)$ are independent uniform random matrices,
and $\ovr{I}$ denotes the identity on $\C^{[3:n]}$.

Let us condition on $Q$ and $R$ for a moment.
By the rotation invariance of the random orthogonal matrix $V$ and of the Gaussian skew-symmetric matrix $S$,
the (conditional) joint distribution of the pair $(V,S)$ is the same as that of $(Q^\tran R V Q, Q^\tran S Q)$.
Therefore, the conditional distribution of $A$ is the same as that of
$$
Q^\tran R V Q D + I + \e Q^\tran S Q
= Q^\tran (R V Q D Q^\tran + I + \e S) Q
=: \widehat{A}.
$$
Let us go back to estimating $p_{1,3}$ in \eqref{eq: p13}. Since $A$ and $\hat{A}$ are identically distributed,
and the event $\EE_S$ does not change when $S$ is replaced by $Q^\tran S Q$,
the conditional probability
$$
\Pr{ \inf_{x \in S_{1,2,3}} \|Ax\|_2 \le 2 \mu \e \amp \|(A_{(1,2,3)})^{-1}\| \le \frac{K_1}{\e} \amp \EE_S \;\Big|\; Q, R}
$$
does not change when $A$ is replaced by $\widehat{A}$.
Taking expectations with respect to $Q$ and $R$ we see that the full (unconditional) probability does not change either, so
\begin{equation}				\label{eq: p13 hat}
p_{1,3} \le \Pr{ \inf_{x \in S_{1,2,3}} \|\widehat{A}x\|_2 \le 2 \mu \e
  \amp \|(\widehat{A}_{(1,2,3)})^{-1}\| \le \frac{K_1}{\e} \amp \EE_S}.
\end{equation}

\subsubsection{Randomizing $D$}

Let us try to understand the terms appearing in $\widehat{A}$.
We think of
$$
\til{D} := Q D Q^\tran
$$
as a randomized version of $D$ obtained by a random change of basis in the first three coordinates.
Then we can express
$$
\widehat{A} = Q^\tran \til{A} Q
\quad \text{where} \quad
\til{A} = R V \til{D} + I + \e S.
$$
Compared to $A = V D + I + \e S$, the random matrix $\til{A}$ incorporates the global perturbations $Q$ and $R$.
Thus we seek to replace $A$ with $\til{A}$ in our problem. To this end, let us simplify two quantities
that appear in \eqref{eq: p13 hat}.

First,
\begin{equation}				\label{eq: A Ahat}
\inf_{x \in S_{1,2,3}} \|\widehat{A}x\|_2 = \inf_{x \in S_{1,2,3}} \|\til{A}x\|_2
\end{equation}
since $Q(S_{1,2,3}) = S_{1,2,3}$ by definition and $Q^\tran \in SO(n)$.
Second, using that $Q$ and $R$ affect only the first three coordinates and since $D$ is diagonal, one checks that
\begin{equation}				\label{eq: Ahat simplified}
\widehat{A}_{(1,2,3)} = \til{A}_{(1,2,3)} = (V \til{D} + I + \e S)_{(1,2,3)} = (V D + I + \e S)_{(1,2,3)}.
\end{equation}
Similarly to previous matrices, we decompose $S$ as
\begin{equation}				\label{eq: S decomposition}
S =
\begin{bmatrix}
S_0 & -Z^\tran \\
Z & \ovr{S}
\end{bmatrix},
\quad \text{where }
S_0 \in \R^{3 \times 3}, \;
\ovr{S} \in \R^{(n-3) \times (n-3)}.
\end{equation}
Note that $S_0$, $\ovr{S}$ and $Z$ are independent, and that $Z \in \R^{(n-3) \times 3}$ is
a random matrix with all i.i.d. $N_\R(0,1)$ entries.

By \eqref{eq: Ahat simplified}, $\widehat{A}_{(1,2,3)}$ is independent of $S_0$, $Z$, $Q$, $R$; it only
depends on $V$ and $\ovr{S}$. Let us condition on $S_0$, $\ovr{S}$ and $V$, and thus fix
$\widehat{A}_{(1,2,3)}$ such that the invertibility condition in \eqref{eq: p13 hat} is satisfied, i.e. such that
$$
\|(\widehat{A}_{(1,2,3)})^{-1}\| \le \frac{K_1}{\e}
$$
(otherwise the corresponding conditional probability is automatically zero).
All randomness remains in the local perturbation $Z$ and the global perturbations $Q$, $R$.

\medskip

Let us summarize our findings. Recalling \eqref{eq: A Ahat}, we have shown that
\begin{equation}							\label{eq: p13 via Atil}
p_{1,3} \le \inf_{S_0, \ovr{S}, V} \Pr[Z, Q, R]{\inf_{x \in S_{1,2,3}} \|\til{A}x\|_2 \le 2 \mu \e \amp \EE_S}
\end{equation}
where
$$
\til{A} = R V \til{D} + I + \e S, \quad \til{D} := Q D Q^\tran,
$$
where $S$ is decomposed as in \eqref{eq: S decomposition},
and where the infimum is over all $\ovr{S}$, $V$ satisfying
\begin{equation}				\label{eq: minor well invertible}
\|(V D + I + \e S)_{(1,2,3)})^{-1}\| \le \frac{K_1}{\e}.
\end{equation}
Compared with \eqref{eq: p13}, we have achieved the following:
we introduced into the problem global perturbations $Q$, $R$ acting on the first three coordinates.
$Q$ randomizes the matrix $D$ and $R$ serves as a further global rotation.

\subsubsection{Reducing to invertibility of random $3 \times 3$ matrices}

Let us decompose the matrix $\til{A} = R V \til{D} + I + \e S$
by revealing its first three rows and columns as before.
To this end, recall that
$R = \bigl[ \begin{smallmatrix}
R_0 & 0 \\
0 & \ovr{I}
\end{smallmatrix} \bigr]$
and
$Q = \bigl[ \begin{smallmatrix}
Q_0 & 0 \\
0 & \ovr{I}
\end{smallmatrix} \bigr]$.
We similarly decompose
$$
V =:
\begin{bmatrix}
V_0 & v \\
u & \ovr{V}
\end{bmatrix}
$$
and
\begin{equation}							\label{eq: D0}
D =:
\begin{bmatrix}
D_0 & 0 \\
0 & \ovr{D}
\end{bmatrix};
\quad \text{then} \quad
\til{D} =
\begin{bmatrix}
\til{D}_0 & 0 \\
0 & \ovr{D}
\end{bmatrix}
\text{ where }
\til{D}_0 := Q_0 D Q_0^\tran.
\end{equation}
Using these and the decomposition of $S$ in \eqref{eq: S decomposition}, we decompose
\begin{equation}				\label{eq: Atil decomposed}
\til{A} = R V \til{D} + I + \e S =
\begin{bmatrix}
R_0 V_0 \til{D}_0 + I_0 + \e S_0  & R_0 v \ovr{D} - \e Z^\tran \\
u \til{D}_0 + \e Z  & \ovr{V} \ovr{D} + \ovr{I} + \e \ovr{S}
\end{bmatrix}
=:
\begin{bmatrix}
H_0 & Y \\
X & \ovr{H}
\end{bmatrix}
\end{equation}
where $I_0$ denotes the identity in $\C^3$.

Note that
$$
\ovr{H}
= \ovr{V} \ovr{D} + \ovr{I} + \e \ovr{S}
= (V D + I + \e S)_{(1,2,3)}
$$
is a well invertible matrix by \eqref{eq: minor well invertible}, namely
\begin{equation}				\label{eq: Hbar invertible}
\|\ovr{H}^{-1}\| \le \frac{K_1}{\e}.
\end{equation}
The next lemma reduces invertibility of $\til{A}$ to invertibility of a $3 \times 3$ matrix $H_0 - Y \ovr{H}^{-1} X$.

\begin{lemma}[Invertibility of a matrix with a well invertible minor]					 \label{lem: well invertible minor}
  Consider a matrix
  $$
  H=
  \begin{bmatrix}
  H_0 & Y \\
  X & \ovr{H}
  \end{bmatrix}
  \quad \text{where } H_0 \in \C^{3 \times 3}, \; \ovr{H} \in \C^{(n-3) \times (n-3)}.
  $$
  Assume that
  $$
  \|\ovr{H}^{-1}\| \le L_1, \quad \|Y\| \le L_2
  $$
  for some $L_1, L_2 > 0$. Then
  $$
  \inf_{x \in S_{1,2,3}} \|Hx\|_2 \ge \frac{1}{\sqrt{n}(1 + L_1 L_2)} \; \smin(H_0 - Y \ovr{H}^{-1} X).
  $$
\end{lemma}

\begin{proof}
Choose $x \in S_{1,2,3}$ which attains $\inf_{x \in S_{1,2,3}} \|Hx\|_2 =: \d$ and decompose it as
$$
x =:
\begin{bmatrix}
x_0 \\ \ovr{x}
\end{bmatrix}
\quad \text{where } x_0 \in \C^{3}, \; \ovr{x} \in \C^{n-3}.
$$
Then
$$
Hx =
\begin{bmatrix}
H_0 x_0 + Y \ovr{x} \\
X x_0 + \ovr{H} \ovr{x}
\end{bmatrix}.
$$
The assumption $\|Hx\|_2 = \d$ then leads to the system of inequalities
$$
\begin{cases}
\|H_0 x_0 + Y \ovr{x}\|_2 \le \d \\
\|X x_0 + \ovr{H} \ovr{x}\|_2 \le \d
\end{cases}
$$
We solve these inequalities in a standard way.
Multiplying the second inequality by $\|\ovr{H}^{-1}\|$, we obtain
$$
\|\ovr{H}^{-1} X x_0 + \ovr{x}\|_2 \le \d \|\ovr{H}^{-1}\| \le \d L_1,
$$
which informally means that $\ovr{x} \approx -\ovr{H}^{-1} X x_0$. Replacing $\ovr{x}$ with $-\ovr{H}^{-1} X x_0$
in the first equation, and estimating the error by the triangle inequality, we arrive at
\begin{align*}
\|H_0 x_0 - Y \ovr{H}^{-1} X x_0\|_2
  &\le \|H_0 x_0 + Y \ovr{x}\|_2 + \|Y \ovr{x} + Y \ovr{H}^{-1} X x_0\|_2 \\
  &\le \d + \|Y\| \cdot \|\ovr{x} + \ovr{H}^{-1} X x_0\|_2 \\
  &\le \d + L_2 \cdot \d L_1
  = \d (1 + L_1 L_2).
\end{align*}
Note that the left hand side is $\|(H_0 - Y \ovr{H}^{-1} X) x_0\|_2$, and that
$\|x_0\|_2 \ge 1/\sqrt{n}$ since $x \in S_{1,2,3}$.
By the definition of the smallest singular value, it follows that
$$
\smin(H_0 - Y \ovr{H}^{-1} X) \le \sqrt{n} \, \d (1 + L_1 L_2).
$$
Rearranging the terms concludes the proof of the lemma.
\end{proof}

In order to apply Lemma~\ref{lem: well invertible minor} for the matrix $\til{A}$ in \eqref{eq: Atil decomposed},
let us check that the boundedness assumptions are satisfied.
We already know that $\|\ovr{H}^{-1}\| \le K_1/\e$ from \eqref{eq: Hbar invertible}.
Further,
$$
\|Y\| = \|R_0 v \ovr{D} - \e Z^\tran\|.
$$
Here, $\|R_0\| = 1$, $\|v\| \le \|V\| \le 1$, $\|\ovr{D}\| \le \|D\| \le K$ by the assumption of the theorem,
and $\|Z\| \le \|S\| \le K_0 \sqrt{n}$ if the event $\EE_S$ holds. Putting these together, we have
$$
\|Y\| \le K + \e K_0 \sqrt{n} \le 2 K
$$
where the last inequality follows from \eqref{eq: K0}.
An application of Lemma~\ref{lem: well invertible minor} yields that, on the event $\EE_S$ one has
\begin{equation}							\label{eq: via YMX}
\inf_{x \in S_{1,2,3}} \|\til{A}x\|_2 \ge \frac{\e}{3 K K_1 \sqrt{n}} \; \smin(H_0 - Y M X)
\end{equation}
where
\begin{equation}				\label{eq: M}
M := \ovr{H}^{-1}, \quad \|M\| \le \frac{K_1}{\e}.
\end{equation}
We have reduced our problem to invertibility of the $3 \times 3$ matrix $H_0 - Y M X$.

\subsubsection{Dependence on the global perturbation $R$}

Let us write our random matrix $H_0 - Y M X$ as a function of the global perturbation $R_0$
(which determines $R$). Recalling \eqref{eq: Atil decomposed},
we have
$$
H_0 - Y M X
  = R_0 V_0 \til{D}_0 + I_0 + \e S_0 - (R_0 v \ovr{D} - \e Z^\tran) M (u \til{D}_0 + \e Z)
  = a + R_0 b,
$$
where
\begin{align}
a &:= I_0 + \e S_0 + \e Z^\tran M u \til{D}_0 + \e^2 Z^\tran M Z, \nonumber\\
b &:= V_0 \til{D}_0 - v \ovr{D} M (u \til{D}_0 + \e Z).			\label{eq: b}
\end{align}
It will be helpful to simplify $a$ and $b$.
We shall first remove the terms $\e S_0$ and $\e^2 Z^\tran M Z$
from $a$, and then (in the next subsection) remove all other terms from $a$ and $b$ 
that depend on $Z$.

To achieve the first step, observe that on the event $\EE_S$, we have
$$
\|\e S_0\| \le \e K_0 \sqrt{n};
$$
\begin{align*}
\|\e^2 Z^\tran M Z\|
  &\le \e^2 \|Z\|^2 \|M\|
  \le \e^2 \|S\|^2 \|M\| \\
  &\le \e^2 \cdot K_0^2 n \cdot \frac{K_1}{\e}
    \quad \text{(by definition of $\EE_S$ and \eqref{eq: M})} \\
  &= \e K_0^2 K_1 n.
\end{align*}
Therefore we can approximate $a$ by the following simpler quantity:
\begin{gather}
a_0 := I_0 + \e Z^\tran M u \til{D}_0, 		\label{eq: a0}\\
\|a - a_0\|_2 \le \e K_0 \sqrt{n} + \e K_0^2 K_1 n
  \le 2 \e K_0^2 K_1 n. \nonumber
\end{gather}
Hence we can replace $a$ by $a_0$ in our problem of estimating
\begin{equation}							\label{eq: YMX via ab}
\smin(H_0 - Y M X) = \smin(a + R_0 b)
\ge \smin(a_0 + R_0 b) - 2 \e K_0^2 K_1 n.
\end{equation}
We have reduced the problem to the invertibility of the $3 \times 3$ random matrix
$a + R_0 b$.

\subsubsection{Removing the local perturbation $Z$}
As we mentioned in the introduction, the argument in this case (when the minor is well invertible)
relies on global perturbations only.
This is the time when we remove the local perturbation $Z$ from our problem.
To this end, we express $a_0 + R_0 b$ as a function of $Z$ using \eqref{eq: a0} and \eqref{eq: b}:
$$
a_0 + R_0 b = L + Z^\tran \e M u \til{D}_0 - R_0 v \ovr{D} M \e Z, \\
$$
where
\begin{equation}								\label{eq: L}
L := I_0 + R_0 (V_0 - v \ovr{D} M u ) \til{D}_0.
\end{equation}
If we condition on everything but $Z$,
we can view $a_0 + R_0 b$ as a Gaussian perturbation of the fixed matrix $L$.
It will then be easy to show that $a_0 + R_0 b$ is well invertible whenever $L$ is.
This will reduce the problem to the invertibility of $L$; the local perturbation $Z$
will thus be removed from the problem.

Formally, let $\l_1 \in (0,1)$ be a parameter to be chosen later; we define the event
$$
\EE_L := \left\{ \smin(L) \ge \l_1 \right\}.
$$
Note that $\EE_L$ is determined by $R_0, Q_0$ and is independent of $Z$.

Let us condition on $R_0, Q_0$ satisfying $\EE_L$.
Then
\begin{equation}								\label{eq: through fZ}
\smin(a_0 + R_0 b) \ge \smin(L) \cdot \smin \left( L^{-1}(a_0 + R_0 b) \right)
\ge \l_1 \cdot \smin(I_0 + f(Z))
\end{equation}
where
\begin{equation}								\label{eq: fZ}
f(Z) := L^{-1} Z^\tran \e M u \til{D}_0 - L^{-1} R_0 v \ovr{D} M \e Z
\end{equation}
is a linear function of (the entries of) $Z$.
A good invertibility of $I_0+f(Z)$ is guaranteed by the following lemma.

\begin{lemma}[Invertibility of Gaussian perturbations]			\label{lem: invertibility Gaussian perturbation}
  Let $m \ge 1$, and let $f : \R^m \to \C^{3 \times 3}$ be a linear (matrix-valued) transformation.
  Assume that $\|f\| \le K$ for some $K \ge 1$, i.e. $\|f(z)\|_\HS \le K \|z\|_2$ for all $z \in \R^m$.
  Let $Z \sim N_\R (0, I_m)$. Then
  $$
  \Pr{\smin(I + f(Z)) \le t} \le C K t^{1/4}, \quad t > 0.
  $$
\end{lemma}

We defer the proof of this lemma to Appendix~\ref{a: invertibility Gaussian perturbation}.

\medskip

 We will use Lemma \ref{lem: invertibility Gaussian perturbation} with $m=3(n-3)$, 
 rewriting the entries of the $(n-3) \times 3$ matrix $Z$ as coordinates of a vector in $\R^m$.
In order to apply this lemma, let us bound $\|f(Z)\|_\HS$ in \eqref{eq: fZ}.
To this end, note that
$\|L^{-1}\| \le \l_1^{-1}$
 if the  event $\EE_L$ occurs;
$\|M\| \le K_1/\e$ by \eqref{eq: M};
$\|u\| \le \|U\| = 1$; $\|v\| \le \|V\| \le 1$;
$\|\til{D}_0\| = \|D_0\| \le \|D\| \le K$;
$\|\ovr{D}\| \le \|D\| \le K$; $\|R_0\| = 1$. It follows that
$$
 \|f(Z)\|_\HS \le 2 \l_1^{-1} K  K_1 \|Z\|_\HS.
$$
An application of Lemma~\ref{lem: invertibility Gaussian perturbation} then yields
$$
\Pr[Z]{\smin(I_0 + f(Z)) \le t} \le C \l_1^{-1} K  K_1  \, t^{1/4}, \quad t > 0.
$$
Putting this together with \eqref{eq: through fZ}, we have shown the following.
Conditionally on $R_0, Q_0$ satisfying $\EE_L$, the matrix $a_0 + R_0 b$ is well invertible:
\begin{equation}								\label{eq: Pr singular}
\Pr[Z]{\smin(a_0 + R_0 b) \le \l_1 t} \le C \l_1^{-1} K K_1  \, t^{1/4}, \quad t > 0.
\end{equation}
This reduces the problem to showing that event $\EE_L$ is likely,
namely that the random matrix $L$ in \eqref{eq: L} is well invertible.
The local perturbation $Z$ has been removed from the problem.

\subsubsection{Invertibility in dimension $3$}  \label{sub: in dimension 3}

In showing that $L$ is well invertible, the global perturbation $R_0$ will be crucial.
Recall that
$$
L = I_0 + R_0 B, \quad \text{where } B = (V_0 - v \ovr{D} M u ) \til{D}_0.
$$
Then $\smin(L) = \smin(B + R_0^{-1})$.
If we condition on everything but $R_0$,
we arrive at the invertibility problem for the perturbation of the fixed matrix $B$
by a random matrix $R_0$ uniformly distributed in $O(3)$.
This is the same kind of problem that our main theorems are about, however
for $3 \times 3$ matrices. But recall that in dimension $3$ the main result has
already been established in Theorem~\ref{thm: main low dim}. It guarantees
that $B+R_0^{-1}$ is well invertible whenever $B$ is not approximately complex orthogonal,
i.e. whenever $\|B B^\tran - I\| \gtrsim \|B\|^2$.
This argument reduces our problem to breaking complex orthogonality for $B$.

Formally, let $\l_2 \in (0,1)$ be a parameter to be chosen later; we define the event
$$
\EE_B := \left\{ \|B B^\tran - I\| \ge \l_2 \|B\|^2 \right\}.
$$
Note that $\EE_B$ is determined by $Q_0$ and is independent of $R_0$.

Let us condition on $Q_0$ satisfying $\EE_B$.
Theorem~\ref{thm: main low dim} then implies that
\begin{equation}								\label{eq: Pr EL}
\P_{R_0} (\EE_L^c)
= \Pr[R_0]{\smin(L) < \l_1}
= \Pr[R_0]{\smin(B + R_0^{-1}) < \l_1}
\le C(\l_1/\l_2)^c.
\end{equation}
This reduces the problem to showing that $\EE_B$ is likely, i.e. that
$B$ is not approximately complex orthogonal.

\subsubsection{Breaking complex orthogonality}

Recall that
$$
B = (V_0 - v \ovr{D} M u ) \til{D}_0
=: T \til{D}_0,
\qquad \til{D}_0 := Q_0 D_0 Q_0^\tran,
$$
where $Q_0$ is a random matrix uniformly distributed in $SO(3)$.
Thus $\til{D}_0$ is a randomized version of $D_0$ obtained by a random change of basis.

Let us condition on everything but $Q_0$, leaving $B$ fixed.
The following general result states that if $D_0$ is not near a multiple of identity,
then $T$ is not approximately complex orthogonal with high probability.

\begin{lemma}[Breaking complex orthogonality]				\label{lem: breaking orthogonality}
  Let $n \in \{2,3\}$.
  Let $D = \diag(d_i) \in \C^{n \times n}$.
  Assume that
  $$
  \max_i |d_i| \le K, \quad |d_1^2 - d_2^2| \ge \d
  $$
  for some $K, \d > 0$.
  Let $T \in \C^{n \times n}$.
  Let $Q$ be uniformly distributed in $SO(n)$ and consider the random matrix
  $B := A Q D Q^\tran$.
  Then
  \begin{equation}								\label{eq: prob BB}
  \Pr{ \|B B^\tran - I\| \le t \|B\|^2 } \le C (t K^2/\d)^c, \quad t > 0.
  \end{equation}
\end{lemma}
We defer the proof of this lemma to Appendix~\ref{a: random basis}.

\medskip

Let us apply Lemma~\ref{lem: breaking orthogonality}  for $D_0 = \diag(d_1, d_2, d_3)$.
Recall that the assumptions of the lemma are satisfied by \eqref{eq: di} and \eqref{eq: di assumption}.
Then an application of the lemma with $t=\l_2$ yields that
\begin{equation}							\label{eq: Pr EB}
\P_{Q_0} (\EE_B^c)
= \Pr[Q_0]{\|B B^\tran - I\| < \l_2 \|B\|^2}
\le C(\l_2 K^2/\d)^c.
\end{equation}
This was the remaining piece to be estimated, and now we can collect all pieces together.

\subsubsection{Putting all pieces together}

By \eqref{eq: Pr EL} and \eqref{eq: Pr EB}, we have
\begin{align*}
\P_{Q_0, R_0} (\EE_L^c)
  &= \E_{Q_0} \P_{R_0} (\EE_L^c | Q_0)
  \le \E_{Q_0} \P_{R_0} (\EE_L^c | Q_0) \; \one_{\{Q_0 \text{ satisfies } \EE_B\}} + \P_{Q_0} (\EE_B^c) \\
  &\le C(\l_1/\l_2)^c + C(\l_2 K^2/\d)^c.
\end{align*}
By a similar conditional argument, this estimate and \eqref{eq: Pr singular} yield
$$
\Pr[Q_0, R_0, Z]{\smin(a_0 + R_0 b) \le \l_1 t \amp \EE_S}
\le q,
$$
where
\begin{equation}							\label{eq: q}
q := C \l_1^{-1} K  K_1  \, t^{1/4} + C(\l_1/\l_2)^c + C(\l_2 K^2/\d)^c.
\end{equation}

Obviously, we can choose $C >1$ and $c<1$.
By \eqref{eq: YMX via ab},
$$
\Pr[Z, Q_0, R_0]{\smin(H_0 - Y M X) < \l_1 t - 2 \e K_0^2 K_1 n \amp \EE_S} \le q
$$
and further by \eqref{eq: via YMX}, we obtain
$$
\Pr[Z, Q_0, R_0]{\inf_{x \in S_{1,2,3}} \|\til{A}x\|_2 < \frac{\e(\l_1 t - 2 \e K_0^2 K_1 n)}{3 K K_1 \sqrt{n}}
  \amp \EE_S} \le q.
$$
Thus we have successfully estimated $p_{1,3}$ in \eqref{eq: p13 via Atil} and in \eqref{eq: p13}:
\begin{equation}				\label{eq: p13 final estimate}
p_{1,3} \le q \quad \text{for } \mu = \frac{\l_1 t - 2 \e K_0^2 K_1 n}{3 K K_1 \sqrt{n}}
\end{equation}
and where $q$ is defined in \eqref{eq: q}.
By an identical argument, the same estimate holds for all $p_{1,i}$ in the sum \eqref{eq: sum p1i}:
\begin{equation}				\label{eq: p1i final estimate}
p_{1,i} \le q, \quad i = 3,\ldots, n.
\end{equation}

Summarizing, we achieved the goal of this section, which was to show that $A$ is well invertible on the set $S_{1,2}$
in the case when there is a well invertible minor $A_{(1,2,i)}$.

\begin{remark}[Doing the same for $(n-2) \times (n-2)$ minors]			 \label{rem: n-2 minors}
  One can carry on the argument of this section in a similar way for $(n-2) \times (n-2)$ minors,
  and thus obtain the same estimate for the probability
  $$
  \Pr{ \inf_{x \in S_{1,2}} \|Ax\|_2 \le 2 \mu \e \amp \|(A_{(1,2)})^{-1}\| \le \frac{K_1}{\e} \amp \EE_S }.
  $$
  as we obtained in \eqref{eq: p13 final estimate} for the probability $p_{1,3}$ in \eqref{eq: p13}.
\end{remark}

\subsection{When all minors are poorly invertible: going $1+2$ dimensions up}		 \label{s: poorly invertible}

In this section we estimate the probability $p_{1,0}$ in the decomposition \eqref{eq: sum p1i}, i.e.
\begin{equation}				\label{eq: p10}
p_{1,0} = \Pr{ \inf_{x \in S_{1,2}} \|Ax\|_2 \le 2 \mu \e \amp \|(A_{(1,2,i)})^{-1}\| > \frac{K_1}{\e}
    \; \forall i \in [3:n] \amp \EE_S }.
\end{equation}

\subsubsection{Invertibility of a matrix with a poorly invertible minor}

The following analog of Lemma~\ref{lem: well invertible minor} for a poorly invertible minor will be helpful
in estimating $p_{1,0}$.
Unfortunately, it only works for $(n-1) \times (n-1)$ minors rather than
$(n-3) \times (n-3)$ or $(n-2) \times (n-2)$ minors.

\begin{lemma}[Invertibility of a matrix with a poorly invertible minor]		 \label{lem: poorly invertible minor}
  Consider an $n \times n$ matrix
  $$
  H=
  \begin{bmatrix}
  H_0 & Y \\
  X & \ovr{H}
  \end{bmatrix}
  \quad \text{where } H_0 \in \C, \; \ovr{H} \in \C^{(n-1) \times (n-1)}.
  $$
  Assume that $X \sim N_\R(\nu, \e^2 I_{n-1})$ for some fixed $\nu \in \C^{n-1}$ and
  $\e>0$.\footnote{Although $X$ is complex-valued, $X-\nu$ is real valued variable distributed
  according to $N(0, \e^2 I_{n-1})$.}
  We assume also that $\ovr{H}$ is a fixed matrix satisfying
  $$
  \|\ovr{H}^{-1}\| \ge L
  $$
  for some $L > 0$, while $H_0$ and $Y$ may be arbitrary,
  possibly random and correlated with $X$. Then
  $$
  \Pr{ \inf_{x \in S_1} \|Hx\|_2 \le \frac{t \e}{\sqrt{n}} - \frac{1}{L} } \le C t \sqrt{n}, \quad t > 0.
  $$
\end{lemma}

\begin{proof}
Choose $x \in S_1$ which attains $\inf_{x \in S_1} \|Hx\|_2 =: \d$ and decompose it as
$$
x =:
\begin{bmatrix}
x_0 \\ \ovr{x}
\end{bmatrix}
\quad \text{where } x_0 \in \C, \; \ovr{x} \in \C^{n-1}.
$$
As in the proof of Lemma~\ref{lem: well invertible minor}, we deduce that
$$
\|\ovr{H}^{-1} X x_0 + \ovr{x}\|_2 \le \d \|\ovr{H}^{-1}\|.
$$
This yields
$$
\|\ovr{x}\|_2 \ge \|\ovr{H}^{-1} X x_0\|_2 - \d \|\ovr{H}^{-1}\|.
$$
Note that
$$
\|\ovr{H}^{-1} X x_0\|_2 = |x_0| \, \|\ovr{H}^{-1} X\| \ge \frac{1}{\sqrt{n}} \, \|\ovr{H}^{-1} X\|,
$$
where the last inequality is due to $x \in S_1$.

Further, we have $\|\ovr{H}^{-1} X\| \sim \|\ovr{H}^{-1}\|_\HS$ by standard concentration techniques.
We state and prove such result in Lemma~\ref{lem: concentration HS} in Appendix~B.
It yields that
$$
\Pr{ \|\ovr{H}^{-1} X\|_2 \le t \e \|\ovr{H}^{-1}\|_\HS } \le C t \sqrt{n}, \quad t > 0.
$$
Next, when this unlikely event does not occur, i.e. when $\|\ovr{H}^{-1} X\|_2 > t \e \|\ovr{H}^{-1}\|_\HS$,
we have
$$
\|\ovr{x}\|_2 \ge \frac{t \e}{\sqrt{n}} \|\ovr{H}^{-1}\|_\HS - \d \|\ovr{H}^{-1}\|
\ge \biggl( \frac{t \e}{\sqrt{n}} - \d \biggr) \|\ovr{H}^{-1}\|
\ge \biggl( \frac{t \e}{\sqrt{n}} - \d \biggr) L.
$$
On the other hand, $\|\ovr{x}\|_2 \le \|x\|_2 = 1$. Substituting and rearranging the terms yields
$$
\d \ge \frac{t \e}{\sqrt{n}} - \frac{1}{L}.
$$
This completes the proof of Lemma~\ref{lem: poorly invertible minor}.
\end{proof}

\subsubsection{Going one dimension up}

As we outlined in Section~\ref{s: intro poorly invertible}, the probability $p_{1,0}$ in \eqref{eq: p10}
will be estimated in two steps. At the first step, which we carry on in this section,
we explore the condition that all $(n-3) \times (n-3)$ minors of $A_{1,2}$ are poorly invertible:
$$
\|(A_{(1,2,i)})^{-1}\| > \frac{K_1}{\e} \quad \forall i \in [3:n].
$$
Using Lemma~\ref{lem: poorly invertible minor} in dimension $n-2$, we will conclude that
the matrix $A_{(1,2)}$ is well invertible on the set of vectors with a large $i$-th coordinate.
Since this happens for all $i$, the matrix $A_{(1,2)}$ is well invertible on all vectors,
i.e. $\|A_{(1,2)}^{-1}\|$ is not too large. This step will thus carry us one dimension up,
from poor invertibility of all minors in dimension $n-3$ to a good invertibility of the minor
in dimension $n-2$.

Since we will be working in dimensions $[3: n]$ during this step, we introduce the appropriate
notation analogous to \eqref{eq: sphere decomposition} restricted to these dimensions.
Thus $S^{[3:n]}$ will denote the unit Euclidean sphere in $\C^{[3:n]}$, so
\begin{equation}				\label{eq: sphere decomposition 3:n}
S^{[3:n]} = \bigcup_{i \in [3:n]} S_i^{[3:n]}
\quad \text{where } S_i^{[3:n]} := \bigl\{ x \in S^{[3:n]} :\; |x_i| \ge n^{-1/2} \bigr\}.
\end{equation}
We apply Lemma~\ref{lem: poorly invertible minor} for
$$
H = A_{(1,2)}, \quad \ovr{H} = A_{(1,2,3)}, \quad L = \frac{K_1}{\e}, \quad t = \frac{2\sqrt{n}}{K_1}.
$$
Recall from Section~\ref{s: local perturbation} that
\[
  A
    = \begin{bmatrix}
   * & * & * &  \hdots  \\
   * & * & * &  \hdots  \\
   * & * & H_0 & Y \\
   \vdots & \vdots & X & \ovr{H}
    \end{bmatrix}
    = VD + I + \e S,
\]
where $S$ is a skew-symmetric Gaussian random matrix (with i.i.d. $N_\R(0,1)$ above-diagonal entries).
Let us condition on everything except the entries $S_{ij}$ with $i \in [4:n]$, $j=3$ and with
$i=3$, $j \in [4:n]$,
since these entries define the parts $X$, $Y$ of $H$.
 Note that $X=\nu+ \e S^{(3)}$, where the vector $\nu \in C^{n-3}$ is independent of $S$, and $S^{(3)}$ is a standard real Gaussian vector with coordinates $S_{4,3}, \ldots, S_{n,3}$.

Lemma~\ref{lem: poorly invertible minor} used with $t=\frac{2 \sqrt{n}}{K_1}$ then implies that if
$\|(A_{(1,2,3)})^{-1}\| > K_1/\e$ then
$$
\Pr[X,Y]{ \inf_{x \in S_3^{[3:n]}} \|A_{(1,2)} x\|_2 \le \frac{\e}{K_1} } \le \frac{Cn}{K_1}.
$$
Therefore, unconditionally,
$$
\Pr{ \inf_{x \in S_3^{[3:n]}} \|A_{(1,2)} x\|_2 \le \frac{\e}{K_1}
  \amp \|(A_{(1,2,3)})^{-1}\| > \frac{K_1}{\e} } \le \frac{Cn}{K_1}.
$$
By an identical argument, the dimension $3$ here can be replaced by any other dimension $i \in [3:n]$.
Using a union bound over these $i$ and \eqref{eq: sphere decomposition 3:n}, we conclude that
$$
\Pr{ \inf_{x \in S^{[3:n]}} \|A_{(1,2)} x\|_2 \le \frac{\e}{K_1}
  \amp \|(A_{(1,2,i)})^{-1}\| > \frac{K_1}{\e} \; \forall i \in [3:n] }
\le \frac{C n^2}{K_1}.
$$
This is of course the same as
\begin{equation}				\label{eq: 1 dim up}
\Pr{ \|(A_{(1,2)})^{-1}\| > \frac{K_1}{\e}
  \amp \|(A_{(1,2,i)})^{-1}\| > \frac{K_1}{\e} \; \forall i \in [3:n] }
\le \frac{C n^2}{K_1}.
\end{equation}
This concludes the first step: we have shown that in the situation of $p_{1,0}$
when all minors $A_{(1,2,i)}$ are poorly invertible, the minor $A_{(1,2)}$ is well invertible.

\subsubsection{Going two more dimensions up}

At the second step, we move from the good invertibility of the minor $A_{(1,2)}$
that we have just established to a good invertibility of the full matrix $A$.
But we have already addressed exactly this problem
in Section~\ref{s: well invertible}, except for the minor $A_{(1,2,3)}$.
So no new argument will be needed in this case.

Formally, combining \eqref{eq: p10}, \eqref{eq: 1 dim up}, and the estimate
\eqref{eq: S norm} on $\EE_S$, we obtain
$$
p_{1,0}
\le \Pr{ \inf_{x \in S_{1,2}} \|Ax\|_2 \le 2 \mu \e \amp \|(A_{(1,2)})^{-1}\| \le \frac{K_1}{\e} \amp \EE_S }
  + \frac{C n^2}{K_1} + 2 \exp(-c K_0^2 n).
$$
The probability here is very similar to the probability $p_{1,3}$ in \eqref{eq: p13}
and is bounded in the same way as in \eqref{eq: p13 final estimate},
see Remark~\ref{rem: n-2 minors}. We conclude that
\begin{equation}							\label{eq: p10 final estimate}
p_{1,0} \le q + \frac{C n^2}{K_1} + 2 \exp(-c K_0^2 n)
\end{equation}
where $\mu$ and $q$ are defined in \eqref{eq: p13 final estimate} and \eqref{eq: q} respectively.

We have successfully estimated $p_{1,0}$ in the sum \eqref{eq: sum p1i}.
This achieves the goal of this section, which was to show that $A$ is well invertible on the set $S_{1,2}$
in the case when there all minors $A_{(1,2,i)}$ are poorly invertible.

\subsection{Combining the results for well and poorly invertible minors}

At this final stage of the proof, we combine the conclusions of Sections~\ref{s: well invertible}
and \ref{s: poorly invertible}.

Recall from \eqref{eq: sum p1i} that
$$
p_1 \le \sum_{i=3}^n p_{1,i} + p_{1,0}.
$$
The terms in this sum were estimated in \eqref{eq: p1i final estimate} and in \eqref{eq: p10 final estimate}.
Combining these, we obtain
$$
p_1 \le n q + \frac{C n^2}{K_1} + 2 \exp(-c K_0^2 n).
$$
An identical argument produces the same estimate for all $p_i$, $i=2,\ldots,n$ in \eqref{eq: sum pi}.
Thus
\begin{align}			\label{eq: p estimated}
p
  &= \Pr{ \smin(D+U) \le \mu \e }
  \le \sum_{i=1}^n p_i + 2 \exp(-c K_0^2 n) \nonumber\\
  &\le n^2 q + \frac{C n^3}{K_1} + 2 (n+1) \exp(-c K_0^2 n).
\end{align}
Recall that $\mu$ and $q$ are defined in \eqref{eq: p13 final estimate} and \eqref{eq: q} respectively, and $C \ge 1$, $c\le 1$ in these inequalities.

Finally, for $t \in (0,1)$, we choose the parameters $K_0 > 1$, $K_1 > 1$, $\e, \l_1, \l_2 \in (0,1)$
to make the expression in \eqref{eq: p estimated} reasonably small.
For example, one can choose
\begin{gather*}
K_0 = \log(1/t), \quad
K_1 = t^{-1/16}, \\
\l_1 = t^{1/16}, \quad
\l_2 = t^{1/32}, \quad
\e = \frac{t^{9/8}}{24 K \log^2(1/t)n^{3/2}}.
\end{gather*}
With this choice, we have
\begin{gather*}
\mu \ge \frac{t^{9/8}}{6 K \sqrt{n}} \ge 2 K_0^2 n \e, \quad
q \le 3 t^{c/32} (K^2/\d)^c, \\
p \le C_1 n^3 t^{c/32} (K^2/\d)^c,
\end{gather*}
and so \eqref{eq: K0} and \eqref{eq: mu is big} are satisfied, and \eqref{eq: C0} is satisfied whenever $t< e^{-C_0}$.
Summarizing, we have shown that
$$
\Pr{ \smin(D+U) \le \frac{t^{9/4}}{144 K^2 \log^4 (1/t) n^2} }
\le  C_1 n^3 t^{c/32} (K^2/\d)^c.
$$
This quickly leads to the conclusion of Theorem~\ref{thm: main orth full}.  \qed

\section{Application to the Single Ring Theorem: proof of Corollary~\ref{cor: single ring}}		\label{s: single ring theorem}

In this section we prove Corollary~\ref{cor: single ring}, which states
that condition (SR3) can be completely eliminated from the Single Ring Theorem.
Let $D_n$ be a sequence of deterministic $n \times n$ diagonal matrices.
(The case of random $D_n$ can be reduced to this by conditioning on $D_n$.)
If $z \neq 0$, then
 \begin{equation}  \label{eq: z out}
  \smin(U_n D_n V_n-z I_n)= |z| \cdot \smin ((1/z) D_n - U_n^{-1}V_n^{-1}),
 \end{equation}
where the matrix $ U_n^{-1}V_n^{-1}$ is uniformly distributed in $U(n)$ or $O(n)$.
Let us first consider the case where the matrices $D_n$ are well invertible,
thus we assume that
$$
r := \inf_{n \in \N} \smin(D_n) > 0.
$$
In the complex case, an application of Theorem~\ref{thm: main unit} yields the inequality
  \begin{equation}  \label{eq: uniform outside 0}
    \Pr{\smin(U_n D_n V_n-z I_n) \le t r} \le t^c n^C,	 \quad 0 \le t < 1/2,
  \end{equation}
which holds (uniformly) for all $z \in \C$, and which implies condition (SR3).
Indeed, Theorem~\ref{thm: main unit} combined with \eqref{eq: z out} imply
the inequality \eqref{eq: uniform outside 0} for $|z| \ge r/2$.
In the disc $|z|< r/2$ we use the trivial estimate
\[
\smin(U_n D_n V_n-z I_n) \ge \smin(U_n D_n V_n)- |z| > r/2,
\]
which again implies \eqref{eq: uniform outside 0}.

Now consider the real case, still under the assumption that $r>0$.
Condition (SR1) allows us to assume that $\|D_n\| \le K$ for some $K$ and for all $n$.
Condition (SR2) and \cite[Lemma 15]{GKZ} imply that  $|s_k(D_n) -1| \ge 1/(4 \k_1)$ for some $1 \le k \le n$.
Hence
\[
\inf_{V \in O(n)} \|D_n -V\| \ge \frac{1}{4 \k_1}.
\]
An application of Theorem \ref{thm: main orth full} together with \eqref{eq: z out} shows
that inequality \eqref{eq: uniform outside 0} holds, which in turn implies condition (SR3).
In this argument, we considered the matrix $(1/z)D_n$, which has {\em complex} entries. This was the reason to prove more general Theorem \ref{thm: main orth full} instead of the simpler Theorem \ref{thm: main orth}.

It remains to analyze the case where the matrices $D_n$ are poorly invertible,
i.e. when $\inf_{n \in \N} \smin(D_n) =0$.
In this case the condition (SR3) can be removed from the Single Ring Theorem
using our results via the following argument, which was communicated
to the authors by Ofer Zeitouni \cite{Z personal}.
The proof of the Single Ring Theorem in \cite{GKZ} uses condition (SR3) only once,
specifically in the proof of \cite[Proposition 14]{GKZ}
which is one of the main steps in the argument.
Let us quote this proposition.

\begin{prop14}[\cite{GKZ}]
  Let $\nu_z^{(n)}$ be the symmetrized\footnote{Symmetrization here means that we consider the set of the
  singular values $s_k$ together with their opposites $-s_k$.}
  empirical measure of the singular values of $U_n D_n V_n-z I_n$.
  Assume that the conditions (SR1), (SR2), and (SR3) of the Single Ring Theorem hold.
  \begin{enumerate}[(i)]
    \item There exists a sequence of events $\Omega_n$ with $\P(\Omega_n) \to 1$
      such that for Lebesgue almost every $z \in \C$, one has
      \begin{equation} \label{cond: SR1}
      \lim_{\e \to 0} \limsup_{n \to \infty}
      \E \int_0^{\e} \mathbf{1}_{\Omega_n} \log |x| \, d \nu_z^{(n)}(x) =0.
      \end{equation}
      Consequently, for almost every $z \in \C$ one has
      \begin{equation} \label{cond: SR2}
      \int_{\R} \log |x| \, d \nu_z^{(n)}(x)
      \to
      \int_{\R} \log |x| \, d \nu_z(x)
      \end{equation}
      for some limit measure $\nu_z$  in probability.
    \item For any $R>0$ and for any smooth deterministic function $\varphi$
    compactly supported in $B_R=\{z \in \C: |z| \le R \}$, one has
        \begin{equation} \label{cond: SR3}
          \int_{\C} \varphi(z) \int_{\R} \log |x| \, d \nu_z^{(n)}(x) \, d m(z)
          \to
          \int_{\C} \varphi(z) \int_{\R} \log |x| \, d \nu_z(x) \, d m(z).
        \end{equation}
  \end{enumerate}
\end{prop14}

Our task is to remove condition (SR3) from this proposition. Since the argument below is the same for unitary and orthogonal matrices, we will not distinguish between the real and the complex case.

\medskip

Even without assuming (SR3),
part (i) can be deduced from Theorems~\ref{thm: main unit} and \ref{thm: main orth}
by the argument of \cite{GKZ}, since condition \eqref{cond: SR1} pertains to a fixed $z$.

It remains to prove (ii) without condition (SR3).
To this end, consider the probability measure $\tilde{\mu}$ with the density
\begin{equation} \label{eq: density  of mu-tilde}
 \frac{d \tilde{\mu}}{d m} (z) = \frac{1}{2 \pi} \D \left(  \int_{\R} \log |x| \, d \nu_z(x) \right).
\end{equation}
This measure was introduced and studied in \cite{GKZ}.
After the Single Ring Theorem is proved it turns out that $\tilde{\mu}=\mu_e$,
where $\mu_e$ is the limit of the empirical measures of eigenvalues.
However, at this point of the proof this identity is not established, so we have to distinguish between these two measures.

It was shown in \cite{GKZ} that for any smooth compactly supported function $f: \C \to \C$ such
that condition (SR3) holds with some $\d, \d'>0$ for almost all $z \in \supp(f)$, one has
    \begin{equation}  \label{eq: weak limit}
     \int_{\C} f(z) \, d \mu_e^{(n)}(z) \to
     \int_{\C} f(z) \, d \tilde{\mu}(z).
    \end{equation}
The argument in the beginning of this section shows that if $Q:=\supp(f) \subset B_R \setminus B_r$
for some $r>0$, then \eqref{eq: uniform outside 0} holds uniformly on $Q$,
and therefore \eqref{eq: weak limit} holds for such $f$.

The proof of \cite[Theorem 1]{GKZ} shows that it is enough to establish (ii)
for all smooth compactly supported functions $\varphi$ that can be represented as $\varphi=\D \psi$,
where $\psi$ is another smooth compactly supported function.
Assume that (ii) fails, thus
there exist $\e>0$, a subsequence $\{n_k\}_{k=1}^{\infty}$, and a  function $\psi: \C \to \C$ as above, such that
   \begin{equation} \label{eq: no convergence}
         \left|  \int_{\C} \D \psi(z) \int_{\R} \log |x| \, d \nu_z^{(n_k)}(x) \, d m(z)
          -
          \int_{\C} \D \psi(z) \int_{\R} \log |x| \, d \nu_z(x) \, d m(z)\right |
          >\e.
   \end{equation}
   Recall the following identity \cite[formula (5)]{GKZ}:
   \begin{equation}				\label{eq: Krishnapur}
    \int_{\C} \psi(z) \, d \mu_e^{(n)}(z) =
    \frac{1}{2 \pi}\int_{\C} \D \psi(z)  \int_{\R} \log |z| \, d \nu_z^{(n)}(x) \, d m(z).
  \end{equation}
    Condition (SR1) implies that the sequence of measures $\mu_e^{(n_k)}$ is tight, so we can extract a further subsequence $\{\mu_e^{(n_{k_l}) } \}_{l=1}^{\infty}$ which converges weakly to a probability measure $\mu$.

    We claim that $\mu=\tilde{\mu}$. Indeed, let $f: \C \to [0,1]$ be a smooth function supported in $B_R \setminus B_r$ for some $r>0$. Then the weak convergence implies
    \[
     \int_{\C} f(z) \, d \mu_e^{(n_{k_l})}(z) \to
     \int_{\C} f(z) \, d \mu(z).
    \]
    Since $f$ satisfies \eqref{eq: weak limit}, we obtain
    \[
          \int_{\C} f(z) \, d \mu(z)
          =     \int_{\C} f(z) \, d \tilde{\mu}(z).
    \]
  This means that the measure $\mu$ coincides with $\tilde{\mu}$ on $\C \setminus \{0\}$.
  Since both $\mu$ and $\tilde{\mu}$ are probability measures,  $\mu=\tilde{\mu}$.

Since $\tilde{\mu}$ is absolutely continuous, we can choose $\t>0$ so that
$\tilde{\mu}(B_{\t})< \frac{\e}{8 \pi \| \psi \|_{\infty}}$.
    Let $\eta: \C \to [0,1]$ be a smooth function such that $\text{supp}(\eta) \subset B_{\t}$ and $\eta(z)=1$ for any $z \in B_{\t/2}$. Then
    \begin{equation}  \label{eq: no mass at 0}
    \int_{\C} \eta(z) \, d \tilde{\mu}(z) <
       \frac{\e}{8 \pi \| \psi \|_{\infty}},
     \end{equation}
      and therefore
    \begin{equation}  \label{eq: sub-subsequence}
      \int_{\C} \eta(z) \, d \mu_e^{(n_{k_l})}(z) <
      \frac{\e}{8 \pi \| \psi \|_{\infty}}
    \end{equation}
    for all sufficiently large $l$.
    Let us estimate the quantity in \eqref{eq: no convergence}:
    \begin{align*}
        &\quad \left|  \int_{\C} \D  \psi (z) \int_{\R} \log |x| \, d \nu_z^{(n_{k_l})}(x) \, d m(z)
          -
          \int_{\C} \D \psi(z) \int_{\R} \log |x| \, d \nu_z(x) \, d m(z)\right | \\
        &\le \left|  \int_{\C} \D \big( (1-\eta) \psi \big) (z) \int_{\R} \log |x| \, d \nu_z^{(n_{k_l})}(x) \, d m(z) \right. \\
          &\hskip 2in  -
          \left. \int_{\C} \D \big( (1-\eta) \psi \big)(z) \int_{\R} \log |x| \, d \nu_z(x) \, d m(z)\right | \\
        &+  \left|  \int_{\C} \D \big( \eta \psi \big) (z) \int_{\R} \log |x| \, d \nu_z^{(n_{k_l})}(x) \, d m(z) \right | \\
          &+
           \left| \int_{\C} \D \big( \eta \psi \big)(z) \int_{\R} \log |x| \, d \nu_z(x) \, d m(z)\right |.
    \end{align*}
    Consider the terms in the right hand side separately. By \eqref{eq: density of mu-tilde} and \eqref{eq: no mass at 0}, we have
    \begin{align*}
      \left| \int_{\C} \D \big( \eta \psi \big)(z) \int_{\R} \log |x| \, d \nu_z(x) \, d m(z)\right |
      &= 2 \pi \left| \int_{\C} \big( \eta \psi \big)(z)  \, d \tilde{\mu}(z)\right | \\
      &\le 2 \pi \|\psi\|_{\infty} \cdot  \left| \int_{\C}  \eta (z)  \, d \tilde{\mu}(z)\right |
      < \frac{\e}{4}.
    \end{align*}
    Similarly, \eqref{eq: Krishnapur} and \eqref{eq: sub-subsequence} imply that for large $l$
    \[
      \left| \int_{\C} \D \big( \eta \psi \big)(z) \int_{\R} \log |x| \, d \nu_z^{(n_{k_l})}(x) \, d m(z)\right |
      < \frac{\e}{4}.
    \]

  The function $\tilde{\varphi}= \D \big( (1-\eta) \psi \big)$ is supported in the annulus $B_R \setminus B_{\t/2}$.
  This function satisfies \eqref{eq: weak limit}, so using \eqref{eq: density of mu-tilde} and \eqref{eq: Krishnapur}, we obtain
  \begin{align*}
      & \int_{\C} \tilde{\varphi} (z) \int_{\R} \log |x| \, d \nu_z^{(n)}(x) \, d m(z)
          -
          \int_{\C} \tilde{\varphi}(z) \int_{\R} \log |x| \, d \nu_z(x) \, d m(z) \\
        & = 2 \pi \int_{\C} \big( (1-\eta) \psi \big) (z) \, d  \mu_e^{(n_{k_l})}(z)
           - 2 \pi \int_{\C} \big( (1-\eta) \psi \big) (z) \, d  \mu(z)
          \to 0.
  \end{align*}
  The combination of these inequalities yields
  \begin{align*}
   &\limsup_{l \to \infty}          \left|  \int_{\C} \D  \psi (z) \int_{\R} \log |x| \, d \nu_z^{(n_{k_l})}(x) \, d m(z)
          -
          \int_{\C} \D \psi(z) \int_{\R} \log |x| \, d \nu_z(x) \, d m(z)\right | \\
   &< \frac{\e}{2},
  \end{align*}
  which contradicts  \eqref{eq: no convergence}. \qed

 \begin{remark}
   Convergence of the empirical measures of eigenvalues $\mu_e^{(n)}$ to the limit measure $\mu_e$
   does not imply the convergence of the eigenvalues to an annulus.
   Indeed, there may be outliers which do not affect the limit measure. For example, assume that $\{D_n\}_{n=1}^{\infty}$ is a sequence of diagonal  matrices
 \[
  D_n= \text{diag}(d_1, \ldots, d_{n-1}, 0),
 \]
 where $d_1, \ldots, d_{n-1}$ are independent random variables uniformly distributed in $[1,2]$. Then the Single Ring Theorem asserts that the support of the measure $\mu_e$ is the annulus $\sqrt{2} \le |z| \le \sqrt{7/3}$, see \eqref{eq: radii of the ring}. At the same time, all matrices $A_n=U_n D_n V_n$ have eigenvalue $0$.

Guionnet and Zeitouni \cite{GZ}  established sufficient conditions for the convergence of the spectrum of $A_n$ to an annulus. Assume that the matrices $A_n$ satisfy (SR1), (SR2), and (SR3), and in addition:

 \begin{enumerate}[(SR4)]
   \item[(SR4)]  Assume that
 \begin{align*}
   &\left ( \int_0^{\infty} x^{-2} \, d \mu_s^{(n)}(x) \right)^{-1/2} \to
    a=\left ( \int_0^{\infty} x^{-2} \, d \mu_s(x) \right)^{-1/2}, \\
  &\left ( \int_0^{\infty} x^{2} \, d \mu_s^{(n)}(x) \right)^{1/2} \to
  b=\left ( \int_0^{\infty} x^{2} \, d \mu_s(x) \right)^{1/2},
 \end{align*}
 and $\inf_n \smin(D_n)>0$ whenever $a > 0$.
 \end{enumerate}
 Then \cite[Theorem 2]{GZ} claims that the spectrum of $A_n$ converges to the annulus $a \le |z| \le b$ in probability.

 Arguing as before, one can eliminate the condition (SR3) from this list. The other conditions are formulated in terms of the matrices $D_n$ only.

 \end{remark}

\appendix

\section{Orthogonal perturbations in low dimensions}  \label{a: main low dim}

In this section we prove Theorem~\ref{thm: main low dim}, which is a
slightly stronger version of the main Theorem~\ref{thm: main orth full}
in dimensions $n=2$ and $n=3$. The argument will be based on Remez-type inequalities.

\subsection{Remez-type ineqalities}

Remez inequality and its variants capture the following phenomenon:
if a polynomial of a fixed degree is small on a set of given measure, then it remains to be small on a larger set (usually an interval). We refer to \cite{Ganzburg, Ganzburg II} to an extensive discussion of these inequalities.

We will use two versions of Remez-type inequalities,
for multivariate polynomials on a convex body and on the sphere.
The first result is due to Ganzburg and Brudnyi \cite{BG 73, BG 76}, see \cite[Section 4.1]{Ganzburg}.

\begin{theorem}[Remez-type inequality on a convex body]				 \label{thm: remez convex}
  Let $V \subset \R^m$ be a convex body, let $E \subseteq V$ be a measurable set,
  and let $f$ be a real polynomial on $\R^m$ of degree $n$. Then
  $$
  \sup_{x \in V} |f(x)| \le \biggl( \frac{4 m |V|}{|E|} \biggr)^n \sup_{x \in E} |f(x)|.
  $$
  Here $|E|$ and $|V|$ denote the $m$-dimensional Lebesgue measures of these sets. \qed
\end{theorem}

The second result can be found in \cite{Ganzburg}, see (3.3) and Theorem~4.2 there.

\begin{theorem}[Remez-type inequality on the sphere]						 \label{thm: remez sphere}
  Let $m \in \{1,2\}$, let $E \subseteq S^m$ be a measurable set,
  and let $f$ be a real polynomial on $\R^{m+1}$ of degree $n$. Then
  $$
  \sup_{x \in S^{m}} |f(x)| \le \biggl( \frac{C_1}{|E|} \biggr)^{2n} \sup_{x \in E} |f(x)|.
  $$
  Here $|E|$ denote the $m$-dimensional Lebesgue measure of $E$. \qed
\end{theorem}

\begin{remark}
  By a simple argument based on Fubini theorem, a similar Remez-type inequality can be proved
  for the real three-dimensional torus $T_3 := S^1 \times S^2 \subset \R^5$ equipped with
  the product measure:
  \begin{equation}				\label{eq: remez torus}
  \sup_{x \in T_3} |f(x)| \le \biggl( \frac{C_1}{|E|} \biggr)^{4n} \sup_{x \in E} |f(x)|.
  \end{equation}
\end{remark}

\subsection{Vanishing determinant}

Before we can prove Theorem~\ref{thm: main low dim}, we
we establish a simpler result, which is deterministic and
which concerns determinant instead of the smallest singular value.
The determinant is simpler to handle because it can be easily expressed
in terms of the matrix entries.

\begin{lemma}[Vanishing determinant]				\label{lem: det small}
  Let $B$ be a fixed $n \times n$ complex matrix, where $n \in \{2,3\}$.
  Assume that $\|B\| \ge 1/2$.
  Let $\e > 0$ and assume that
  $$
  |\det(B+U)| \le \e \quad \text{for all } U \in SO(n).
  $$
  Then
  $$
  \|B B^\tran - I\| \le C \e \|B\|.
  $$
\end{lemma}

\begin{proof}
To make this proof more readable, we will write $a \lesssim b$ if $a \le C b$ for a suitable absolute constant $C$,
and $a \approx_\e b$ if $|a-b| \lesssim \e$.

{\bf Dimension $n=2$.}
Let us represent
$$
U = U(\phi) =
\begin{bmatrix}
  \cos \phi & \sin \phi \\
  -\sin \phi & \cos \phi
\end{bmatrix}.
$$
Then $\det(B+U)$ is a trigonometric polynomial
$$
\det(B+U) = k_0 + k_1 \cos \phi + k_2 \sin \phi
$$
whose coefficients can be expressed in terms the coefficients of $B$:
$$
k_0 = \det(B) + 1; \quad
k_1 = B_{11} + B_{22}; \quad
k_2 = B_{12} - B_{21}.
$$
By assumption, the modus of this trigonometric polynomial is bounded by $\e$.
Therefore all of its coefficients are also bounded, i.e.
$$
|k_i| \lesssim \e, \quad i = 1,2,3.
$$

It is enough to check that all entries of $B B^\tran$ are close to the corresponding
entries of $I$. We will check this for entries $(1,1)$ and $(1,2)$; others are similar.
Then
$$
(B B^\tran)_{11} = B_{11}^2 + B_{12}^2
\approx_{\e'} -B_{11} B_{22} + B_{12} B_{21}
$$
where we used that $|k_1| \lesssim \e$, $|k_2| \lesssim \e$, and thus
the resulting error can be estimated as
$$
\e' \lesssim \e ( |B_{11}| + |B_{12}| ) \lesssim \e \|B\|.
$$
But
$$
-B_{11} B_{22} + B_{12} B_{21} = -\det(B) \approx_\e 1,
$$
where we used that $|k_0| \lesssim \e$.
We have shown that
$$
|(B B^\tran)_{11} - 1| \lesssim \e \|B\| + \e \lesssim \e \|B\|,
$$
as required.

Similarly we can estimate
$$
(B  B^\tran)_{12} = B_{11} B_{21} + B_{12} B_{22}
  \approx_{\e'} B_{11} B_{12} - B_{12} B_{11}
  = 0.
$$

Repeating this procedure for all entries, we have shown that
$$
|(B B^\tran)_{ij} - I_{ij}| \lesssim \e \|B\|
$$
for all $i,j$.
This immediately implies the conclusion of the lemma in dimension $n=2$.

\medskip

{\bf Dimension $n=3$.}
We claim that
\begin{equation}				\label{eq: entries dets}
\det(B) \approx_\e -1; \quad
B_{ij} \approx_\e (-1)^{i+j+1} \det(B^{ij}), \quad i, j \in \{1,2,3\},
\end{equation}
where $B^{ij}$ denotes the minor obtained by removing the $i$-th row and $j$-th column from $B$.

Let us prove \eqref{eq: entries dets} for $i=j=1$; for other entries the argument is similar.
Let
$$
U = U(\phi) =
\begin{bmatrix}
  1 & 0 & 0 \\
  0 & \cos \phi & \sin \phi \\
  0 & -\sin \phi & \cos \phi
\end{bmatrix}.
$$
Then as before, $\det(B+U)$ is a trigonometric polynomial
$$
\det(B+U) = k_0 + k_1 \cos \phi + k_2 \sin \phi
$$
whose coefficients can be expressed in terms the coefficients of $B$.
Our argument will only be based on the free coefficient $k_0$, which one can
quickly show to equal
$$
k_0 = \det
\begin{bmatrix}
  B_{11} + 1 & B_{12} & B_{13} \\
  B_{21} & B_{22} & B_{23} \\
  B_{31} & B_{32} & B_{33}
\end{bmatrix}
+ B_{11} + 1
= \det(B) + \det(B^{11}) + B_{11} + 1.
$$
As before, the assumption yields that $|k_0| \lesssim \e$, so
\begin{equation}				\label{eq: B11 minor 1}
\det(B) + \det(B^{11}) + B_{11} + 1 \approx_\e 0.
\end{equation}

Repeating the same argument for
$$
U = U(\phi) =
\begin{bmatrix}
  -1 & 0 & 0 \\
  0 & \cos \phi & \sin \phi \\
  0 & \sin \phi & -\cos \phi
\end{bmatrix}
$$
yields
\begin{equation}				\label{eq: B11 minor 2}
\det(B) - \det(B^{11}) - B_{11} + 1 \approx_\e 0.
\end{equation}
Estimates \eqref{eq: B11 minor 1} and \eqref{eq: B11 minor 2}  together imply that
$$
\det(B) \approx_\e -1; \quad
B_{11} \approx_\e -\det(B^{11}).
$$
This implies claim \eqref{eq: entries dets} for $i=j=1$; for other entries the argument
is similar.

Now we can estimate the entries of $B^\tran B$. Indeed, by \eqref{eq: entries dets} we have
\begin{align}
(B B^\tran)_{11} &= B_{11}^2 + B_{12}^2 + B_{13}^2 \nonumber\\
  &\approx_{\e'} -B_{11} \det(B^{11}) + B_{12} \det(B^{12}) - B_{13} \det(B^{13}),		\label{eq: det expansion}
\end{align}
where the error $\e'$ can be estimated as
$$
\e' \lesssim \e ( |B_{11}| + |B_{12}| + |B_{13}|) \lesssim \e \|B\|.
$$
Further, the expression in \eqref{eq: det expansion} equals $-\det(B)$, which can be
seen by expanding the determinant along the first row. Finally, $-\det(B) \approx_\e 1$
by \eqref{eq: entries dets}. We have shown that
$$
|(B^\tran B)_{11} - 1| \lesssim \e \|B\| + \e \lesssim \e \|B\|,
$$
as required.

Similarly we can estimate
\begin{align*}
(B B^\tran)_{12} &= B_{11} B_{21} + B_{12} B_{22} + B_{13} B_{23} \\
  &\approx_{\e'} B_{11} \det(B^{21}) - B_{12} \det(B^{22}) + B_{13} \det(B^{23}) \\
  &=B_{11} (B_{12} B_{33} - B_{32} B_{13})
    - B_{12} ( B_{11} B_{33} - B_{31} B_{13})
    + B_{13} ( B_{11} B_{32} - B_{31} B_{12} ) \\
  &=0
\end{align*}
(all terms cancel).

Repeating this procedure for all entries, we have shown that
$$
|(B B^\tran)_{ij} - I_{ij}| \lesssim \e \|B\|
$$
for all $i,j$.
This immediately implies the conclusion of the lemma in dimension $n=3$.
\end{proof}

\subsection{Proof of Theorem~\ref{thm: main low dim}}

Let us fix $t$; without loss of generality, we can assume that $t < \d/100$.
Let us assume that $B+U$ is poorly invertible with significant probability:
\begin{equation}				\label{eq: poor invertibility}
\Pr{\smin(B+U) \le t} > p(\d, t).
\end{equation}
where $p(\d,t) \in (0,1)$ is to be chosen later.
Without loss of generality we may assume that $U$ is distributed
uniformly in $SO(n)$ rather than $O(n)$. Indeed, since $O(n)$ can be decomposed into two
conjugacy classes $SO(n)$ and $O(n) \setminus SO(n)$,
the inequality \eqref{eq: poor invertibility}
must hold over at least one of these classes. Multiplying one of the rows of $B+U$ by $-1$
if necessary, one can assume that it holds for $SO(n)$.

Note that $\|B\| \ge 1/2$; otherwise
$\smin(B+U) \ge 1 - \|B\| \ge 1/2 > t $ for all $U \in O(n)$, which violates \eqref{eq: poor invertibility}.

\subsubsection{Dimension $n=2$.}
In this case the result follows easily from Lemma~\ref{lem: det small}
and Remez inequality.
Indeed, the event $\smin(B+U) \le t$ implies
$$
|\det(B+U)| = \smin(B+U) \|B+U\| \le t (\|B\|+1) \le 3 t \|B\|.
$$
Therefore, by \eqref{eq: poor invertibility} we have
\begin{equation}								\label{eq: det D+U small}
\Pr{\det(B+U) \le 3 t \|B\|} > p(\d, t).
\end{equation}
A random uniform rotation
$U = \left[ \begin{smallmatrix} x & y \\ -y & x \end{smallmatrix} \right] \in SO(2)$
is determined by a random uniform point $(x,y)$ on the real sphere $S^1$.
Now, $\det(D+U)$ is a complex-valued quadratic polynomial in variables $x,y$ that is restricted to the real sphere $S^1$.
Hence $|\det(D+U)|^2$ is a real-valued polynomial of degree $4$ restricted to the real sphere $S^1$.
Therefore, we can apply the Remez-type inequality, Theorem~\ref{thm: remez sphere},
for
the subset $E := \{ U: |\det(D+U)|^2/\|B\|^2 \le 3t \}$ of $S^1$ which satisfies $|E| \ge 2 \pi \, p(\d, t)$
according to \eqref{eq: det D+U small}.
It follows that
$$
|\det(B+U)| \le \Big( \frac{C_1}{p(\d, t)} \Big)^{C_0} t \|B\| \quad \text{for all } U \in SO(2).
$$
An application of Lemma~\ref{lem: det small} then gives
$$
\|B B^\tran  - I\| \le C_2 \Big( \frac{C_1}{p(\d, t)} \Big)^{C_0} t \|B\|^2.
$$
On the other hand, assumption \eqref{eq: no complex orthogonality}
states that the left hand side is bounded below by $\d \|B\|^2$.
It follows that
\begin{equation}				\label{eq: p delta t}
\d \le C_2 \Big( \frac{C_1}{p(\d, t)} \Big)^{C_0} t.
\end{equation}

Now we can choose $p(\d,t) = C(t/\d)^c$ with sufficiently large absolute constant $C$
and sufficiently small absolute constant $c>0$ so that inequality
\eqref{eq: p delta t} is violated.
Therefore \eqref{eq: poor invertibility} fails with this choice of $p(\d,t)$,
and consequently we have
$$
\Pr{\smin(B+U) \le t} \le C(t/\d)^c,
$$
as claimed.

\subsubsection{Dimension $n=3$: middle singular value} 		\label{s: middle}
This time, determinant is the product of three singular values.
So repeating the previous argument would produce an extra factor of $\|B\|$,
which would force us to require that
$$
\|B B^\tran - I\| \ge \d \|B\|^3.
$$
instead of \eqref{eq: no complex orthogonality}.

The weak point of this argument is that it ignores the {\em middle} singular value of $B$,
replacing it by the largest one. We will now be more careful.
Let $s_1 \ge s_2 \ge s_3 \ge 0$ denote the singular values of $B$.

Assume the event $\smin(B+U) \le t$ holds.
Since $\|U\|=1$, the triangle inequality, Weyl's inequality and the assumption imply
that the three singular values of $B+U$ are bounded one by
$s_1+1 \le \|B\| + 1 \le 3\|B\|$, another by $s_2+1$ and the remaining one by $t$.
Thus
$$
|\det(B+U)| \le 3 t (s_2+1) \|B\|.
$$

Let $K \ge 2$ be a parameter to be chosen later.
Suppose first that $s_2 \le K$ holds. Then $|\det(B+U)| \le 6 t K \|B\|$,
and we shall apply Remez inequality. In order to do this, we can
realize $U \in SO(3)$ as a random uniform
rotation of the $(x,y)$ plane followed by an independent rotation that maps the $z$ axis
to a uniform random direction.
Thus $U$ is determined by a random point $(x,y,z_1,z_2,z_3)$ in the real three-dimensional torus
$T_3 = S^1 \times S^2$, chosen according to the uniform (product) distribution.
Here $(x,y) \in S^1$ and $(z_1,z_2,z_3) \in S^2$ determine the two rotations we described
above.\footnote{This construction and its higher-dimensional generalization
  follow the 1897 description of the Haar measure on $SO(n)$ by Hurwitz, see \cite{DS}.}

We regard $|\det(B+U)|^2$ as a real polynomial in five variables $x,y,z_1,z_2,z_3$ and
constant degree, which is restricted to $T_3$.
Thus we can apply the Remez-type inequality for the torus, \eqref{eq: remez torus},
and an argument similar to the case $n=2$ yields
\begin{equation}				\label{eq: if s2 small}
\Pr{\smin(B+U) \le t} \le C(tK/\d)^c.
\end{equation}

Now we assume that $s_2 \ge K$.
We will show that, for an appropriately chosen $K$, this case is impossible,
i.e. $B+U$ can not be poorly invertible with considerable probability.

\subsubsection{Reducing to one dimension}
Since $s_1 \ge s_2 \ge K \ge 2$, it must be that $s_3 \le 2$;
otherwise all singular values of $B$ are bounded
below by $2$, which clearly implies that $\smin(B+U) \ge 1$ for all $U \in O(3)$.
This will allow us to reduce our problem to one dimension.
To this end, we consider the singular value decomposition of $B$,
$$
B = s_1 q_1 p_1^* +  s_2 q_2 p_2^* + s_3 q_3 p_3^*,
$$
where $\{p_1, p_2, p_3\}$ and $\{q_1, q_2, q_3\}$ are orthonormal bases in $\C^3$.

Assume the event $\smin(B+U) \le t$ holds. Then there exists
$x \in \C^3$, $\|x\|_2 = 1$, such that
$$
\|(B+U)x\|_2 \le t.
$$
We are going to show that $x$ is close to $p_3$, up to a unit scalar factor.
To see this, note that $\|Bx\|_2 \le 1+ t \le 2$, so
\begin{align}
4 &\ge \|Bx\|_2^2
  = s_1^2 |p_1^*x|^2 + s_2^2 |p_2^*x|^2 + s_3^2 |p_3^*x|^2
  \ge K^2 ( |p_1^*x|^2 + |p_2^*x|^2 ) 			\label{eq: proj p1 p2}\\
  &= K^2 (1 - |p_3^*x|^2). \nonumber
\end{align}
It follows that
\begin{equation}				\label{eq: proj p3}
1 - \frac{4}{K^2} \le |p_3^*x| \le 1
\end{equation}
(the right hand side holds since $\|p_3^*\|_2 = \|x\|_2 = 1$.)
Let $\eta := p_3^*x / |p_3^*x|$; then
\begin{align*}
\|x - \eta p_3\|_2^2
  &= \|x/\eta - p_3\|_2^2
  = |p_1^*(x/\eta-p_3)|^2 + |p_2^*(x/\eta-p_3)|^2 + |p_3^*(x/\eta-p_3)|^2 \\
  &= |p_1^*x|^2 + |p_2^*x|^2 + \big| |p_3^*x| - 1 \big|^2 		
      \quad \text{(by orthogonality and definition of $\eta$)}\\
  &\le \frac{4}{K^2} + \frac{16}{K^4} \quad \text{(by \eqref{eq: proj p1 p2} and \eqref{eq: proj p3})}\\
  &\le \frac{8}{K^2}.
\end{align*}

Now, by triangle inequality,
\begin{equation}				\label{eq: q3Bp3}
|q_3^*(B+U) p_3|
  = |q_3^*(B+U) \eta p_3|
  \le |q_3^*(B+U) x| + |q_3^*(B+U) (x - \eta p_3)|.
\end{equation}
The first term is bounded by $\|q_3^*\|_2 \|(B+U)x\|_2 \le t$.
The second term is bounded by
$$
\|q_3^*(B+U)\|_2 \|x - \eta p_3\|_2
\le (\|q_3^*B\|_2 + 1) \frac{\sqrt{8}}{K}
= (s_3 + 1) \frac{\sqrt{8}}{K}
\le \frac{3\sqrt{8}}{K} \le \frac{9}{K}.
$$
Therefore the expression in \eqref{eq: q3Bp3} is bounded by $t + 9/K$.

Summarizing, we have found vectors $u, v \in \C^3$, $\|u\|_2 = \|v\|_2 = 1$, such that
the event $\smin(B+U) \le t$ implies
\[
|u^\tran (B+U) v| \le t + 9/K.
\]
Note that the vectors $u = (q_3^*)^\tran$, $v=p_3$ are fixed; they depend on $B$ only.
By \eqref{eq: poor invertibility}, we have shown that
$$
\Pr{ |u^\tran (B+U) v| \le t + 9/K } \ge p(\d,t).
$$
We can apply Remez inequality for $|u^\tran (B+U) v|^2$, which is a quadratic
polynomial in the entries of $U$. It yields
\begin{equation}				\label{eq: uBv prelim}
|u^\tran (B+U) v| \le \Big( \frac{C_1}{p(\d,t)} \Big)^{C_0} (t + 9/K)
\quad \text{for all } U \in SO(3).
\end{equation}

Let $c_0 \in (0,1)$ be a small absolute constant.
Now we can choose
\begin{equation}				\label{eq: p K}
p(\d,t) = C(t/\d)^c, \quad K = 4 (\d/t)^{1/2}
\end{equation}
with sufficiently large absolute constant $C$
and sufficiently small absolute constant $c>0$ so that
the right hand side in \eqref{eq: uBv prelim} is bounded by $c_0$.
Summarizing, we have shown that
\begin{equation}				\label{eq: uBv}
|u^\tran (B+U) v| \le c_0
\quad \text{for all } U \in SO(3).
\end{equation}
We are going to show that this is impossible.
In the remainder of the proof, we shall write $a \ll 1$ to mean that
$a$ can be made arbitrarily small by a suitable choice of $c_0$,
i.e. that $a \le f(c_0)$ for some fixed real valued positive function (which does not depend
on anything) and such that $f(x) \to 0$ as $x \to 0_+$.

\medskip

\subsubsection{Testing on various $U$}
Let us test \eqref{eq: uBv} on
$U = U(\phi) =
\left[ \begin{smallmatrix}
  \cos \phi & \sin \phi & 0 \\
  -\sin \phi & \cos \phi & 0 \\
  0 & 0 & 1
\end{smallmatrix} \right]$.
Writing the bilinear form as a function of $\phi$, we obtain
$$
u^\tran (B+U) v = k + (u_1 v_1 + u_2 v_2) \cos \phi + (u_1 v_2 - u_2 v_1) \sin \phi
$$
where $k = k(B, u, v)$ does not depend on $\phi$.
Since this trigonometric polynomial is small for all $\phi$,
its coefficients must are also be small, thus
$$
|u_1 v_1 + u_2 v_2| \ll 1, \quad |u_1 v_2 - u_2 v_1| \ll 1.
$$
We can write this in terms of a matrix-vector product as
$$
\left\| \begin{bmatrix}
  u_1 & u_2 \\
  -u_2 & u_1
\end{bmatrix}
\begin{bmatrix}
  v_1 \\
  v_2
\end{bmatrix} \right\|_2 \ll 1.
$$
Since $c_0$ is small, it follows that either the matrix
$\left[\begin{smallmatrix}
  u_1 & u_2 \\
  -u_2 & u_1
\end{smallmatrix}\right]$
is poorly invertible (its smallest singular value is small), or the vector
$\left[\begin{smallmatrix}
  v_1 \\
  v_2
\end{smallmatrix}\right]$
has small norm. Since $\|u\|_2 = 1$, the norm of the matrix is bounded by $\sqrt{2}$.
Hence poor invertibility of the matrix is equivalent to smallness of its determinant,
which is $u_1^2 + u_2^2$. Formally, we conclude that
\begin{equation}				\label{eq: v or u}
\text{either } |v_1|^2 + |v_2|^2 \ll 1
\quad \text{or } |u_1^2 + u_2^2| \ll 1.
\end{equation}

Assume that $|v_1|^2 + |v_2|^2 \ll 1$; since $\|v\|_2 = 1$ this implies
$|v_3| \ge 1/2$.
Now test \eqref{eq: uBv} on
$U = U(\phi) =
\left[ \begin{smallmatrix}
  \cos \phi & 0 & \sin \phi \\
  0 & 1 & 0 \\
  -\sin \phi & 0 & \cos \phi
\end{smallmatrix} \right]$
a similar argument yields
$$
|u_1 v_1 + u_3 v_3| \ll 1, \quad |u_1 v_3 - u_3 v_1| \ll 1.
$$
Since $|u_1| \le 1$, $|u_3| \le 1$, $|v_1| \ll 1$ and $|v_3| \ge 1/2$,
this system implies
$$
|u_1| \ll 1, \quad |u_3| \ll 1.
$$
Similarly, testing on
$U = U(\phi) =
\left[ \begin{smallmatrix}
  1 & 0 & 0 \\
  0 & \cos \phi & \sin \phi \\
  0 & -\sin \phi & \cos \phi
\end{smallmatrix} \right]$,
the same argument yields
$$
|u_2| \ll 1, \quad |u_3| \ll 1.
$$
So we proved that $|u_1| \ll 1$, $|u_2| \ll 1$, $|u_3| \ll 1$.
But this is impossible since $\|u\|_2 = 1$.

We have thus shown that in \eqref{eq: v or u} the first option never holds,
so the second must hold. In other words, we have deduced from \eqref{eq: uBv}
that
\begin{equation}				\label{eq: u1 u2}
|u_1^2 + u_2^2| \ll 1.
\end{equation}
Using a similar argument (for rotations $U$ in coordinates $1, 3$ and $2, 3$)
we can also deduce that
\begin{equation}				\label{eq: u1 u2 u3}
|u_1^2 + u_3^2| \ll 1, \quad |u_2^2 + u_3^2| \ll 1.
\end{equation}	
Inequalities \eqref{eq: u1 u2} and \eqref{eq: u1 u2 u3} imply that
$$
|u_1^2| \ll1, \quad |u_2^2| \ll1, \quad |u_3^2| \ll1.
$$
But this contradicts the identity $\|u\|_2 = 1$.

This shows that \eqref{eq: uBv} is impossible, for a suitable choice of
absolute constant $c_0$.

\medskip

\subsubsection{Conclusion of the proof}
Let us recall the logic of the argument above.
We assumed in \eqref{eq: poor invertibility} that $B+U$ is poorly invertible
with significant probability, \linebreak
$\Pr{\smin(B+U) \le t} > p(\d, t)$.
With the choice $p(\d,t) = C(t/\d)^c$, $K = 4 (\d/t)^{1/2}$ made in \eqref{eq: p K},
we showed that either \eqref{eq: if s2 small} holds (in the case $s_2 \le K$), i.e.
$\Pr{\smin(B+U) \le t} \le C(tK/\d)^c$, or a contradiction appears
(in the case $s_2 \ge K$).
Therefore, one always has
$$
\Pr{\smin(B+U) \le t} \le \max(p(\d,t), \ C(tK/\d)^c).
$$
Due to our choice of $p(\d,t)$ and $K$, the right hand side is
bounded by $C(tK/\d)^{c/2}$.
This completes the proof of Theorem~\ref{thm: main low dim}.
\qed

\section{Some tools used in the proof of Theorem~\ref{thm: main orth full}}	\label{a: tools}

In this appendix we shall prove auxiliary results used in the proof of Theorem~\ref{thm: main orth full}.
These include: Lemma~\ref{lem: concentration HS} on small ball probabilities for Gaussian random
vectors (which we used in the proof of Lemma~\ref{lem: poorly invertible minor}),
Lemma~\ref{lem: invertibility Gaussian perturbation} on invertibility of Gaussian perturbations,
and Lemma~\ref{lem: breaking orthogonality} on breaking complex orthogonality by a random change
of basis.
Some of the proofs of these results follow standard arguments, but the statements are
difficult to locate in the literature.

\subsection{Small ball probabilities}

\begin{lemma}			 \label{lem: conc rv}
  Let $X \sim N_\R(\mu, \s^2)$ for some $\mu \in \R$, $\s>0$. Then
  $$
  \Pr{ |X| \le t \s } \le t, \quad t>0.
  $$
\end{lemma}

\begin{proof}
The result follows since the density of $X$ is bounded by $1/\s\sqrt{2\pi}$.
\end{proof}

\begin{lemma}						\label{lem: concentration HS}
Let $Z \sim N_\R(\mu, \s^2 I_n)$ for some $\mu \in \C^n$
and $\s > 0$.\footnote{This means that $X-\mu$ is real valued variable distributed
  according to $N(0, \s^2 I_{n-1})$.}
 Then
$$
\Pr{ \|MZ\|_2 \le t \s \|M\|_\HS } \le C t \sqrt{n}, \quad t>0.
$$
\end{lemma}

\begin{proof}
By rescaling we can assume that $\s = 1$.

First we give the argument in the real case, for $\mu \in \R^n$, $M \in \R^{n \times n}$.
Let $M_i^\tran$ denote the $i$-th row of $M$, and let $\mu = (\mu_1, \ldots, \mu_n)$.
Choose $i \in [n]$ such that $\|M_i\|_2 \ge \|M\|_\HS/\sqrt{n}$.
Note that $M_i^\tran Z \sim N_\R(\nu_i, \|M_i\|_2^2)$ for some $\nu_i \in \R^n$.
Lemma~\ref{lem: conc rv} yields that
$$
\Pr{ |M_i^\tran Z| \le \tau \|M_i\|_2 } \le C t, \quad t>0.
$$
Since $\|MZ\|_2 \ge |M_i^\tran Z|$ and $\|M_i\|_2 \ge \|M\|_\HS/\sqrt{n}$,
this quickly leads to the completion of the proof.

The complex case can be proved by decomposing $\mu$ and $M$
into real and imaginary parts, and applying the real version
of the lemma to each part separately.
\end{proof}

\subsection{Invertibility of random Gaussian perturbations}				 \label{a: invertibility Gaussian perturbation}

In this appendix we prove Lemma~\ref{lem: invertibility Gaussian perturbation}.

First we note that without loss of generality, we can assume that $m=18$.
Indeed, since $f$ is linear it can be represented as
$$
f(z) = [f(z)_{ij}]_{i,j=1}^3 = [a_{ij}^\tran z + \sqrt{-1} \, b_{ij}^\tran z]_{i,j=1}^3,
$$
where $a_{ij}$ and $b_{ij}$ are some fixed vectors in $\R^m$.
By rotation invariance of $Z$, the joint distribution of the Gaussian
random variables $a_{ij}^\tran Z$ and $b_{ij}^\tran Z$ is determined by
the inner products of the vectors $a_{ij}$ and $b_{ij}$. There are $18$ of
these vectors; so we can isometrically realize them in $\R^{18}$.
It follows that the distribution of $f(Z)$ is preserved, and thus we can assume that $m=18$.

Let $R \ge 1$ be a parameter to be chosen later.
By a standard Gaussian concentration inequality,
$\|Z\|_2 \le R$ with probability at least $1 - 2 \exp(-c R^2)$.
On this event, the matrix in question is well bounded:
$\|I + f(Z)\| \le 1 + \|f(Z)\|_\HS \le 2 K R$, and consequently we have
$$
|\det(I+f(Z))| \le \smin(I+f(Z)) \cdot (2 K R)^2.
$$
Therefore we can estimate the probability in question as follows:
\begin{multline}				\label{eq: det small split}
\Pr{\smin(I+f(Z)) \le t} \\
\le \Pr{|\det(I+f(Z))| \le  (2 K R)^2 t, \; \|Z\| \le R} + 2 \exp(-c R^2).
\end{multline}

Since $f$ is linear, $|\det(I+f(Z))|^2$ is a real polynomial in $Z \in \R^{18}$ of degree $6$,
and thus we can apply Remez inequality, Theorem~\ref{thm: remez convex}.
We are interested in the Gaussian measure of the set
$$
E := \{ Z \in \R^{18} :\; |\det(I+f(Z))| \le  (2 K R)^2 t, \; \|Z\| \le R \}
$$
which is a subset of
$$
V := \{ Z \in \R^{18} :\; \|Z\| \le R \}.
$$
The conclusion Theorem~\ref{thm: remez convex} is in terms of the Lebesgue rather than Gaussian measures of these sets:
$$
|\det(I+MZ)|^2 \le \biggl( \frac{C_1 |V|}{|E|} \biggr)^6 \cdot ((2 K R)^2 t)^2
\quad \text{for all } Z \in V.
$$
Taking square roots and substituting $Z = 0$ in this inequality, we obtain
$$
1 \le \biggl( \frac{C_1 |V|}{|E|} \biggr)^3 \cdot  (2 K R)^2 t,
$$
thus
$$
|E| \le C_1 |V| \cdot ( (2 K R)^2 t)^{1/3}
\le C_2 R^{18} \cdot ( (2 K R)^2 t)^{1/3},
$$
where the last inequality follows from the definition of $V$.
Further, note that the (standard) Gaussian measure of $E$ is bounded by the Lebesgue measure $|E|$,
because the density is bounded by density $(2 \pi)^{-9} \le 1$.
Recalling the definition of $E$, we have shown that
$$
\Pr{|\det(I+f(Z))| \le  (2 K R)^2 t, \; \|Z\| \le R} \le C_2 R^{18} \cdot ( (2 K R)^2 t)^{1/3}.
$$
Substituting this back into \eqref{eq: det small split}, we obtain
$$
\Pr{|\det(I+f(Z)) \le t} \le C_2 R^{18} \cdot ( (2 K R)^2t)^{1/3} + 2 \exp(-c R^2).
$$
Finally, we can optimize the parameter $R \ge 1$, choosing for example $R = t^{-1/1000}$
to conclude that
$$
\Pr{|\det(I+f(Z)) \le t} \le C_3 K^{2/3} t^{1/4}.
$$
This completes the proof of Lemma~\ref{lem: invertibility Gaussian perturbation}.
\qed

\subsection{Breaking complex orthogonality}								\label{a: random basis}

In this section we prove Lemma~\ref{lem: breaking orthogonality} about
breaking complex orthogonality by a random change of basis.

We will present the argument in dimension $n=3$; the dimension $n=2$ is very similar.
Without loss of generality, we can assume that $t < 1/2$.
Note that by assumption,
$$
\|B\| \le \|T\| \|D\| \le K \|T\|.
$$
Then the probability in the left side of \eqref{eq: prob BB} is bounded by
$$
\Pr{ \|B B^\tran - I\| \le K^2 \|T\|^2 t }
= \Pr { \|\widehat{T} \widehat{T}^\tran - I\| \le K^2 \|T\|^2 t },
\quad \text{where } \widehat{T} = TQD.
$$
We can pass to Hilbert-Schmidt norms (recall that all matrices are $3 \times 3$ here)
and further bound this probability by
$$
\P \big\{ \|\widehat{T} \widehat{T}^\tran - I\|_\HS \le 3 K^2 \|T\|_\HS^2 t \big\}.
$$

Assume the conclusion of the lemma fails, so this probability is
larger than $C (t K^2/\d)^c$. We are going to apply Remez inequality
and conclude that $\|B B^\tran - I\|$ is small with probability one.
Recalling the Hurwitz description of a uniform random rotation $Q \in SO(3)$
which we used in Section~\ref{s: middle}, we can parameterize $Q$ by a uniform random point
on the real torus $T_3 = S^1 \times S^2 \subset \R^5$.
Under this parametrization,
$\big( \|\widehat{T} \widehat{T}^\tran - I\|_\HS / 3 K^2 \|T\|_\HS^2 \big)^2$
becomes a polynomial in five variables and with constant degree restricted to $T_3$.

Our assumption above is that this polynomial is bounded by $t^2$ on a subset of $T_3$
of measure larger than $C (t K^2/\d)^c$.
Then the Remez-type inequality for the torus \eqref{eq: remez torus}
implies that the polynomial is bounded
on the entire $T_3$ by
$$
\Big( \frac{C_1}{C (t K^2/\d)^c} \Big)^{C_0} t^2
\le \Big( \frac{\d}{10^4 K^2} \Big)^2
$$
where the last inequality follows by a suitable choice of a large absolute constant $C$ and
a small absolute constant $c$ in the statement of the lemma.
This means that
$$
\|\widehat{T} \widehat{T}^\tran - I\|_\HS
\le 3 K^2 \|T\|_\HS^2 \cdot \frac{\d}{10^4 K^2}
\le \frac{\d}{500} \|T\|_\HS^2
\quad \text{for all } Q \in SO(3).
$$

There is an entry of $T$ such that $|T_{ij}| \ge \frac{1}{3} \|T\|_\HS$.
Since the conclusion of the Lemma is invariant under permutations of the rows of $T$,
we can permute the rows in such a way that $T_{ij}$ is on the diagonal, $i=j$.
Furthermore, for simplicity we can assume that $i=j=1$; the general case is similar.
We have
\begin{equation}				\label{eq: That entry}
|(\widehat{T} \widehat{T})_{11}^\tran - 1|
\le \frac{\d}{500} \|T\|_\HS^2
\quad \text{for all } Q \in SO(3).
\end{equation}

We shall work with $Q$ of the form $Q = Q_1 Q_2$ where $Q_1, Q_2 \in SO(3)$.
We shall use $Q_1$  to mix the entries of $T$ and $Q_2$
to test the inequality \eqref{eq: That entry}.
Let $Q_1 = Q_1(\phi) = \left[ \begin{smallmatrix}
  \cos \phi & \sin \phi & 0 \\
  -\sin \phi & \cos \phi & 0 \\
  0 & 0 & 1
\end{smallmatrix} \right]$, $\phi \in [0, 2\pi]$, and consider the matrix
$$
G := TQ_1.
$$
Since $G_{11} = T_{11} \cos \phi - T_{12} \sin \phi$ and
$G_{12} = T_{11} \sin \phi + T_{12} \cos \phi$,
one can find $\phi$ (and thus $Q_1$) so that
\begin{equation}				\label{eq: G11 G12}
|G_{11}^2 - G_{12}^2| \ge \frac{1}{9} |T_{11}|^2 \ge \frac{1}{81} \|T\|_\HS^2.
\end{equation}

Recall that $\widehat{T} = T Q_1 Q_2 D = G Q_2 D$.
Substituting into inequality \eqref{eq: That entry}
$Q_2 = \left[ \begin{smallmatrix}
  1 & 0 & 0 \\
  0 & 1 & 0 \\
  0 & 0 & 1
\end{smallmatrix} \right]$
and
$Q_2 = \left[ \begin{smallmatrix}
  0 & 1 & 0 \\
  1 & 0 & 0 \\
  0 & 0 & -1
\end{smallmatrix} \right]$, we obtain
\begin{gather*}
  |d_1^2 G_{11}^2 + d_2^2 G_{12}^2 + d_3^2 G_{13}^2 - 1| \le \frac{\d}{500} \|T\|_\HS^2; \\
  |d_1^2 G_{12}^2 + d_2^2 G_{11}^2 + d_3^2 G_{13}^2 - 1| \le \frac{\d}{500} \|T\|_\HS^2.
\end{gather*}
We subtract the second inequality from the first and conclude that
\begin{equation}								\label{eq: dG}
|(d_1^2 - d_2^2) (G_{11}^2 - G_{12}^2)| \le \frac{\d}{250} \|T\|_\HS^2.
\end{equation}

On the other hand, recall that $|d_1^2 - d_2^2| \ge \d$ by assumption
and $|G_{11}^2 - G_{12}^2| \ge \frac{1}{81} \|T\|_\HS^2$ by \eqref{eq: G11 G12}.
Hence $|(d_1^2 - d_2^2) (G_{11}^2 - G_{12}^2)| \ge \frac{\d}{81} \|T\|_\HS^2$.
This contradicts \eqref{eq: dG}. The proof of Lemma~\ref{lem: breaking orthogonality} is complete.
\qed

\end{document}